\documentclass[12pt,letterpaper]{article}

\pdfoutput=1

\usepackage{graphics,epsfig,graphicx,color}
\graphicspath{{/EPSF/}{../Figures/}{Figures/}}
\usepackage{subfigure}

\usepackage{amssymb,latexsym}

\usepackage{setspace}

\usepackage{amsmath}

\usepackage{mathrsfs,amsfonts}

\usepackage{amsthm,amsxtra}

\usepackage{mathtools}

\usepackage[final]{showkeys}

\usepackage{stmaryrd}

\usepackage{algorithm}
\usepackage{algorithmicx}
\usepackage{algpseudocode}

\usepackage[openbib]{currvita}

\newif\ifPDF
\ifx\pdfoutput\undefined
	\PDFfalse
\else
	\ifnum\pdfoutput > 0
		\PDFtrue
	\else
		\PDFfalse
	\fi
\fi

\ifPDF
	\usepackage{pdftricks}
	\begin{psinputs}
		\usepackage{pstricks}
		\usepackage{pstcol}
		\usepackage{pst-plot}
		\usepackage{pst-tree}
		\usepackage{pst-eps}
		\usepackage{multido}
		\usepackage{pst-node}
		\usepackage{pst-eps}
	\end{psinputs}
\else
	\usepackage{pstricks}
\fi

\ifPDF
	\usepackage[debug,pdftex,colorlinks=true, 
	linkcolor=blue, bookmarksopen=false,
	plainpages=false,pdfpagelabels]{hyperref}
\else
	\usepackage[dvips]{hyperref}
\fi

\usepackage{fancyhdr}

\usepackage{comment}

\usepackage[toc,page]{appendix}

\usepackage{cancel}

\newtheorem{theorem}{Theorem}[section]
\newtheorem{lemma}[theorem]{Lemma}
\newtheorem{definition}[theorem]{Definition}

\newtheorem{proposition}[theorem]{Proposition}

\newcommand{\dsum}{\displaystyle\sum}

\newcommand{\dint}{\displaystyle\int}

\newcommand{\eps}{\varepsilon}

\newcommand{\bbC}{\mathbb C}

\newcommand{\bbR}{\mathbb R} \newcommand{\bbS}{\mathbb S}

\newcommand{\bnu}{{\boldsymbol \nu}}

\newcommand{\brho}{\mbox{\boldmath{$\rho$}}}
\newcommand{\bxi}{\boldsymbol \xi}

\newcommand{\be}{\mathbf e}

 \newcommand{\bp}{\mathbf p}

 \newcommand{\bv}{\mathbf v} 
 \newcommand{\bx}{\mathbf x} 
\newcommand{\by}{\mathbf y}

 \newcommand{\bV}{\mathbf V}

\newcommand{\cA}{\mathcal A} \newcommand{\cB}{\mathcal B}
\newcommand{\cC}{\mathcal C}  
\newcommand{\cE}{\mathcal E} 
 
\newcommand{\cI}{\mathcal I} 
 \newcommand{\cL}{\mathcal L}
 
 \newcommand{\cP}{\mathcal P} 
 \newcommand{\cR}{\mathcal R}
\newcommand{\cS}{\mathcal S} 
 
\newcommand{\cW}{\mathcal W}

\newenvironment{keywords}
{\noindent{\bf Key words.}\small}{\par\vspace{1ex}}
\newenvironment{AMS}
{\noindent{\bf AMS subject classifications 2010.}\small}{\par}

\newcommand{\wt}{\widetilde}
\newcommand{\wh}{\widehat}
\newcommand{\bzeta}{\boldsymbol\zeta}
\newcommand{\fki}{\mathfrak{i}}

\title{Nonlinear quantitative photoacoustic tomography with two-photon absorption}

\author{
	Kui Ren\thanks{
		Department of Mathematics and ICES,
		University of Texas, Austin, TX 78712; 
		ren@math.utexas.edu .
	}
	\and 
	Rongting Zhang\thanks{
		Department of Mathematics, 
		University of Texas, Austin, TX 78712; 
		rzhang@math.utexas.edu .
	}
}
    
\begin{document}

\maketitle



\begin{abstract}
Two-photon photoacoustic tomography (TP-PAT) is a non-invasive optical molecular imaging modality that aims at inferring two-photon absorption property of heterogeneous media from photoacoustic measurements. In this work, we analyze an inverse problem in quantitative TP-PAT where we intend to reconstruct optical coefficients in a semilinear elliptic PDE, the mathematical model for the propagation of near infra-red photons in tissue-like optical media with two-photon absorption, from the internal absorbed energy data. We derive uniqueness and stability results on the reconstructions of single and multiple optical coefficients, and present some numerical reconstruction results based on synthetic data to complement the theoretical analysis.
\end{abstract}


\begin{keywords}
	Photoacoustic tomography (PAT), two-photon PAT (TP-PAT), two-photon absorption, hybrid inverse problems, semilinear diffusion equation, numerical reconstruction.
\end{keywords}


\begin{AMS}
	 35R30, 49N45, 65M32, 74J25.
\end{AMS}


\section{Introduction}
\label{SEC:Intro}

Two-photon photoacoustic tomography (TP-PAT)~\cite{LaLeChChSu-OE14,LaBoGrBuBe-OE13,TsSoOmNt-OL14,UrYiYaZh-JBO14,WiYoNoInChChPaSh-Optica14,YaNaTa-PPUIS10,YaNaTa-OE11,YaTa-PPUIS09,YeKo-OE10} is an imaging modality that aims at reconstructing optical properties of heterogeneous media using the photoacoustic effect resulted from two-photon absorption. Here by two-photon absorption we mean the phenomenon that an electron transfers to an excited state after simultaneously absorbing two photons whose total energy exceed the electronic energy band gap. The main motivation for developing two-photon PAT is that two-photon optical absorption can often be tuned to be associated with specific molecular signatures, such as in stimulated Raman photoacoustic microscopy, to achieve label-free molecular imaging. Therefore, TP-PAT can be used to visualize particular cellular functions and molecular processes inside biological tissues.

The principle of TP-PAT is the same as that of the regular PAT~\cite{Beard-IF11,CoLaBe-SPIE09,LiWa-PMB09,Wang-IEEE08}, except that the photoacoustic signals in TP-PAT are induced via two-photon absorption in addition to the usual single-photon absorption. In TP-PAT, we send near infra-red (NIR) photons into an optically absorbing and scattering medium, for instance a piece of biological tissue, $\Omega \subseteq\bbR^n$ ($n \ge 2$), where they diffuse. The density of the photons, denoted by $u(\bx)$, solves the following semilinear diffusion equation:
\begin{equation}\label{EQ:Diff TP}
	\begin{array}{rcll}
  	-\nabla\cdot \gamma(\bx) \nabla u(\bx) + \sigma(\bx) u(\bx) + \mu(\bx)|u|u(\bx) &=& 0, &\mbox{in}\ \ \Omega\\
       u(\bx) &=& g(\bx), &\mbox{on}\ \ \partial\Omega
	\end{array}
\end{equation}
where the function $g(\bx)$ models the incoming NIR photon source, the function $\gamma(\bx)$ is the diffusion coefficient of the medium, $\sigma(\bx)$ is the usual single-photon absorption coefficient of the medium, and $\mu(\bx)$ is the intrinsic two-photon absorption coefficient. The total two-photon absorption coefficient is given by the product $\mu(\bx) |u|$ where the absolute value operation is taken to ensure that the total two-photon absorption coefficient is non-negative, a property that needs to be preserved for the diffusion model~\eqref{EQ:Diff TP} to correctly reflect the physics.

The medium absorbs a portion of the incoming photons and heats up due to the absorbed energy. The heating then results in thermal expansion of the medium. The medium cools down after the photons exit. This cooling process results in contraction of the medium. The expansion-contraction of the medium generates ultrasound waves. The process is called the photoacoustic effect. The initial pressure field generated by the photoacoustic effect can be written as~\cite{FiScSc-PRE07}
\begin{equation}\label{EQ:Data Q TP}
	H(\bx) = \Gamma(\bx) \Big[\sigma(\bx) u(\bx) + \mu(\bx) |u| u(\bx)\Big], \qquad \bx\in \Omega.
\end{equation}
where $\Gamma$ is the Gr\"uneisen coefficient that describes the efficiency of the photoacoustic effect. 
This initial pressure field generated by single-photon and two-photon absorption processes evolves, in the form of ultrasound, according to the classical acoustic wave equation~\cite{BaUh-IP10,FiScSc-PRE07}.

The data we measure in TP-PAT are the ultrasound signals on the surface of the medium. From these measured data, we are interested in reconstructing information on the optical properties of the medium. The reconstruction is usually done in two steps. In the first step, we reconstruct the initial pressure field $H$ in~\eqref{EQ:Data Q TP} from measured data. This step is the same as that in a regular PAT, and has been studied extensively in the past decade; see, for instance, ~\cite{AmBrJuWa-LNM12,BuMaHaPa-PRE07,CoArBe-IP07,FiHaRa-SIAM07,HaScSc-M2AS05,Hristova-IP09,KiSc-SIAM13,KuKu-EJAM08,Nguyen-IPI09,StUh-IP09} and references therein. In the second step of TP-PAT, we attempt to reconstruct information on the optical coefficients, for instance, the two-photon absorption coefficient $\mu$, from the result of the first step inversion, i.e. the internal datum $H$ in~\eqref{EQ:Data Q TP}. This is called the quantitative step in the regular PAT~\cite{AmBoJuKa-SIAM10,BaRe-CM11,BaUh-IP10,CoArKoBe-AO06,GaOsZh-LNM12,LaCoZhBe-AO10,MaRe-CMS14,PuCoArKaTa-IP14,ReGaZh-SIAM13,ReZhZh-IP15,SaTaCoAr-IP13,Zemp-AO10}.

Due to the fact that the intrinsic two-photon absorption coefficient is very small, it is generally believed that events of two-photon absorption in biological tissues can only happen when the local photon density is sufficiently high (so that the total absorption $\mu|u|$ is large enough). In fact, the main difficulty in the development of TP-PAT is to be able to measure the ultrasound signal accurate enough such that the photoacoustic signal due to two-photon absorption is not completely buried by noise in the data. In recent years, many experimental research have been conducted where it is shown that the effect of two-photon absorption can be measured accurately; see, for instance, the study on the feasibility of TP-PAT on various liquid samples in ~\cite{YaNaTa-PPUIS10,YaNaTa-OE11,YaTa-PPUIS09} (solutions), ~\cite{LaLeChChSu-OE14,YaTa-PPUIS09} (suspensions) and~\cite{LaBoGrBuBe-OE13} (soft matter).

Despite various experimental study of TP-PAT, a thorough mathematical and numerical analysis of the inverse problems in the second step of TP-PAT is largely missing, not to mention efficient reconstruction algorithms. The objective of this study is therefore to pursue in these directions. In the rest of the paper, we first recall in Section~\ref{SEC:Model} some fundamental mathematical results on the properties of solutions to the semilinear diffusion equation~\eqref{EQ:Diff TP}. We then develop in Section~\ref{SEC:Absorption} the theory of reconstructing the absorption coefficients. In Section~\ref{SEC:MultiPara} we analyze the linearized problem of simultaneously reconstructing the absorption coefficients and the diffusion coefficient. Numerical simulations are provided in Section~\ref{SEC:Num} to validate the mathematical analysis and demonstrate the quality of the reconstructions. Concluding remarks are offered in Section~\ref{SEC:Concl}.

\section{The semilinear diffusion model}
\label{SEC:Model}

To prepare for the study of the inverse coefficient problems, we recall in this section some general results on the semilinear diffusion model~\eqref{EQ:Diff TP}. Thanks to the absolute value operator in the quadratic term $\mu|u| u$ in the equation, we can follow the standard theory of calculus of variation, as well as the theory of generalized solutions to elliptic equations in divergence form, to derive desired properties of the solution to the diffusion equation that we will need in the following sections. The results we collected here are mostly minor modifications/simplifications of classical results in~\cite{AmMa-Book07,BaSe-Book11,Evans-Book10,GiTr-Book00}. We refer interested readers to these references, and the references therein, for more technical details on these results.

We assume, in the rest of the paper, that the domain $\Omega$ is smooth and satisfies the usual exterior cone condition~\cite{GiTr-Book00}. We assume that all the coefficients involved are bounded in the sense that there exist positive constants $\theta\in\bbR$ and $\Theta\in\bbR$ such that
\begin{equation}\label{EQ:Coeff Assump A}
	0<\theta \leq \Gamma(\bx), \gamma(\bx), \sigma(\bx), \mu(\bx) \leq \Theta <\infty,\ \forall \bx\in\bar\Omega.
\end{equation}
Unless stated otherwise, we assume also that
\begin{equation}\label{EQ:Coeff Assump B}
	(\gamma, \sigma, \mu) \in [W^{1,2}(\bar\Omega)]^3,\quad\mbox{and},\quad \mbox{$g(x)$ is the restriction of a $\cC^3(\bar\Omega)$ function on $\partial\Omega$} .
\end{equation}
where $W^{1,2}(\Omega)$ denotes the usual Hilbert space of $L^2(\Omega)$ functions whose first weak derivative is also in $L^2(\Omega)$. Note that we used $W^{1,2}(\Omega)$ instead of $H^1(\Omega)$ to avoid confusion with the $H$ we used to denote the internal data in~\eqref{EQ:Data Q TP}.

Technically speaking, in some of the results we obtained below, we can relax part of the above assumptions. However, we will not address this issue at the moment. For convenience, we define the function $f(\bx,z)$ and the linear operator $\cL$,
\begin{equation}\label{EQ:F L}
	f(\bx,z) = \sigma(\bx) z + \mu(\bx) |z| z,\quad\mbox{and}\quad \cL u = -\nabla \cdot \gamma \nabla u.
\end{equation}
With our assumptions above, it is clear that $\cL$ is uniformly elliptic, and $f(\bx,z)$ is continuously differentiable with respect to $z$ on $\bar\Omega\times\bbR$. Moreover, $f_z(\bx,z):=\partial_z f(\bx,z)=\sigma(\bx)+2\mu(\bx)|z|\ge\theta>0$, $\forall z\in\bbR$.

We start by recalling the definition of weak solutions to the semilinear diffusion equation~\eqref{EQ:Diff TP}. We say that $u\in \cW\equiv\{w|w \in W^{1,2}(\Omega)\ \mbox{and}\ w_{|\partial\Omega} =g\}$ is a weak solution to~\eqref{EQ:Diff TP} if
\begin{equation*}
	\int_{\Omega} \gamma(\bx) \nabla u\cdot \nabla v + \sigma(\bx) u(\bx) v(\bx) + \mu(\bx) |u|u(\bx) v(\bx) d\bx =0,\ \ \forall v\in W_{0}^{1,2}(\Omega). 
\end{equation*}
We first summarize the results on existence, uniqueness and regularity of the solution to~\eqref{EQ:Diff TP} in the following lemma.
\begin{lemma}\label{thm:regularity}
Let $(\gamma, \sigma, \mu)$ satisfy~\eqref{EQ:Coeff Assump A}, and assume that $g\in \cC^{0}(\partial \Omega)$. Then there is a unique weak solution $u\in W^{1,2}(\Omega)$ such that $u\in \cC^{\alpha}(\Omega)\cap \cC^{0}(\bar{\Omega})$ for some $0<\alpha<1$. If we assume further that $(\gamma, \sigma, \mu)$ and $g$ satisfy~\eqref{EQ:Coeff Assump B}, then $u\in W^{3,2}(\Omega)\cap \cC^{0}(\bar{\Omega})$.
\end{lemma}
\begin{proof}
This result is scattered in a few places in~\cite{AmMa-Book07,BaSe-Book11} (for instance~\cite[Theorem 1.6.6]{BaSe-Book11}). We provide a sketch of proof here. For any function $w\in \cW$, we define the following functional associated with the diffusion equation~\eqref{EQ:Diff TP}:
\[
I[w] =\int_{\Omega} L(\bx,w, Dw)d\bx = \int_{\Omega} \left[\frac{1}{2}\gamma|\nabla w|^2 +\frac{1}{2} \sigma w^2 + \frac{1}{3}\mu |w|w^2\right] d\bx .
\]
It is straightforward to verify that $I[w]: \cW\to \bbR$ is strictly convex (thanks again to the absolute value in the third term) and differentiable on $\cW$ with
\[
I'[w]v = \int_{\Omega} \Big[\gamma(\bx)\nabla w\cdot\nabla v + \sigma(\bx) w v + \mu(\bx) |w|w v \Big] d\bx .
\]
We also verify that the function $L(\bx, z, \bp)$ satisfies the following growth conditions:
\begin{align*}
|L(\bx,z,\bp)| &\leq C(1+|z|^3+|\bp|^2),\\
|D_{z}L(\bx,z,\bp)| &\leq C(1+|z|^2),\\
|D_{\bp}L(\bx,z,\bp)| &\leq C(1+|\bp|),
\end{align*}
for all $\bx\in \Omega$, $z\in \bbR$ and $\bp\in \bbR^n$. It then follows from standard results in calculus of variations~\cite{AmMa-Book07,BaSe-Book11,Evans-Book10} that there exists a unique $u\in \cW$ satisfies
\begin{equation*}
I[u] =\min_{w\in\cW} I[w],
\end{equation*}
and $u$ is the unique weak solution of~\eqref{EQ:Diff TP}. By Sobolev embedding, when $n=2,3$, there exists $q>n$, such that $u \in L^{q}(\Omega)$. This then implies that $f(\bx, u) \in L^{q/2}(\Omega)$ with the assumption~\eqref{EQ:Coeff Assump A}. Let us rewrite the diffusion equation~\eqref{EQ:Diff TP} as
\begin{equation}\label{EQ:Diff TP New}
-\nabla\cdot(\gamma \nabla u) = f(\bx, u),\ \ \mbox{in}\ \Omega, \qquad 
u = g,\ \ \mbox{on}\ \partial \Omega.
\end{equation}
Following standard results in~\cite{Evans-Book10,GiTr-Book00}, we conclude that $f\in L^{q/2}(\Omega)$ implies $u\in\cC^{\alpha}(\Omega)$ for some $0<\alpha<1$, where $\alpha = \alpha(n,\Theta/\theta)$. Moreover, when $g\in \cC^{0}(\partial \Omega)$, $u\in \cC^{0}(\bar{\Omega})$. If we assume further that $(\gamma, \sigma, \mu)$ and $g$ satisfy~\eqref{EQ:Coeff Assump B}, then $f\in W^{1,2}$ thanks to the fact that $u\in\cC^0(\bar\Omega)$. Equation~\eqref{EQ:Diff TP New} then implies that $u\in W^{3,2}(\Omega)\cap \cC^{0}(\bar{\Omega})$~\cite{Evans-Book10,GiTr-Book00}.
\end{proof}

We now recall the following comparison principle for the solutions to the semilinear diffusion equation~\eqref{EQ:Diff TP}.
\begin{proposition}\label{PROP:Comparison}
	(i) Let $u,v\in W^{1,2}(\Omega)\cap \cC^{0}(\bar{\Omega})$ be functions such that $\cL u+f(\bx, u) \le 0$ and $\cL v + f(\bx, v) \ge 0$ in $\Omega$, and $u\leq v$ on $\partial \Omega$. Then $u\leq v$ in $\Omega$. (ii) If, in addition, $\Omega$ satisfies the exterior cone condition or $u,v\in W^{2,2}(\Omega)$, then either $u\equiv v$ or $u<v$.
\end{proposition}
\begin{proof}
For $t\in[0, 1]$, let $u_t = tu + (1-t)v$ and define $a(\bx)=\dint_{0}^{1} f_{z}(u_t,\bx) dt$. It is then straightforward to check that $a(\bx)\ge \theta> 0$ (since $f_z\ge \theta>0$). With the assumption that $u\in \cC^{0}(\bar{\Omega})$ and $v\in \cC^{0}(\bar{\Omega})$, we conclude that $u_t$ is bounded from above when $t\in[0, 1]$. Therefore, $a(\bx)\le \Lambda<\infty$ for some $\Lambda>0$. We also verify that $f(u,\bx)- f(v,\bx) =a(\bx)(u-v)$. Let $w=u-v$, we have, from the assumptions in the proposition, that
\begin{equation*}
	\cL w + a(\bx) w \le 0, \quad \mbox{in}\ \Omega, \qquad w \le 0,\quad \mbox{on}\ \partial\Omega.
\end{equation*}
Since $\cL+a$ is uniformly elliptic, by the weak maximum principle for weak solutions~\cite[Theorem 8.1]{GiTr-Book00}, $w\le 0$ in $\Omega$. This then implies that $u\leq v$ in $\Omega$.

If we assume in addition that $u,v\in W^{2,2}(\Omega)$, we can use the strong maximum principle to conclude that $w\equiv 0$ if $w(0)=0$ for some $x\in\Omega$. Therefore, either $w\equiv 0$, in which case $u=v$, or $w<0$, in which case $u<v$. If $u, v \in W^{1,2}(\Omega)$ and $\Omega$ satisfies the exterior cone condition, we can use~\cite[Theorem 8.19]{GiTr-Book00} to draw the same conclusion.
\end{proof}

The above comparison principle leads to the following assertion on the solution to the semilinear diffusion equation~\eqref{EQ:Diff TP}.
\begin{proposition}\label{PROP:Maximum}
	Let $u_j$ be the solution to~\eqref{EQ:Diff TP} with boundary condition $g_j$, $j=1,2$. Assume that $\gamma$, $\sigma$, $\mu$ and $\{g_j\}_{j=1}^2$ satisfy the assumptions in~\eqref{EQ:Coeff Assump A} and~\eqref{EQ:Coeff Assump B}. Then the following statements hold: (i) if $g_j\ge 0$, then $u_j\ge 0$; (ii) $\sup_{\Omega} u_j\le \sup_{\partial\Omega} g_j$; (iii) if $g_1> g_2$, then $u_1(\bx)>u_2(\bx)$ $\forall \bx\in\Omega$.
\end{proposition}
\begin{proof}
(i) follows from the comparison principle in Proposition~\ref{PROP:Comparison} and the fact that $u\equiv 0$ is a solution to~\eqref{EQ:Diff TP} with homogeneous Dirichlet condition $g=0$. (ii) By (i), $u_j\ge 0$. Therefore $f(\bx,u_j)\ge 0$. Therefore, we can have
\begin{equation*}
	-\nabla\cdot(\gamma \nabla u_j) = - f(\bx,u_j)\le 0,\quad \text{in}\ \Omega.
\end{equation*}
By the maximum principle, $\sup_\Omega u_j \le \sup_{\partial\Omega} g_j$. (iii) is a direct consequence of part (ii) of Proposition~\ref{PROP:Comparison}.
\end{proof}

In the study of the inverse problems in the next sections, we sometimes need the solution to the semilinear diffusion equation to be bounded away from $0$. We now prove the following result.
\begin{theorem}\label{THM:Positivity}
	Let $u$ be the solution to~\eqref{EQ:Diff TP} generated with source $g\ge\eps>0$ for some $\eps$. Then there exists $\eps'>0$ such that $u\ge \eps'>0$.
\end{theorem}
\begin{proof}
We follow the arguments in~\cite{AlDiFrVe-AdM17}. We again rewrite the PDE as
\[
	-\nabla\cdot\gamma\nabla u=-f(\bx, u),\quad \mbox{in}\ \Omega, \qquad u=g,\quad \mbox{on}\ \partial\Omega.
\]
Then by classical gradient estimates, see for instance~\cite[Proposition 2.20]{HaLi-Book97}, we know that there exists $K>0$, depending on $\gamma$, $|\nabla \gamma|$ and $\Omega$, such that
\[
	|u(\bx)-u(\bx_0)| \le K|\bx-\bx_0|,\quad \forall \bx\in\Omega,\quad \bx_0\in\partial\Omega.	
\]
Using the fact that $g\ge \eps$, we conclude from this inequality that there exists a $d>0$ such that 
\[
u(\bx)\ge \eps/2,\qquad \forall \bx\in\Omega\backslash\Omega_d,
\]
where $\Omega_d=\{\bx\in\Omega:\ {\rm dist}(\bx,\partial\Omega)>d\}$. Therefore, $\sup_{\Omega_{d/2}} u\ge \eps/2$.

Let $c(\bx)=\sigma(\bx)+\mu(\bx)|u(\bx)|$. Due to the fact that $u$ is nonnegative and bounded from above, we have that $0<\theta \le c(\bx)\le\Theta(1+\sup_{\partial\Omega}|g|)$. We then have that $u$ solves
\[
-\nabla\cdot\gamma\nabla u+c u =  0,\quad \mbox{in}\ \Omega, \qquad u=g,\quad \mbox{on}\ \partial\Omega.
\]
By the Harnack inequality (see~\cite[Corollary 8.21]{GiTr-Book00}), we have that there exists a constant $C$, depending on $d$, $\gamma$, $c$, $\Omega$, and $\Omega_{d/2}$, such that
\[
	C \inf_{\Omega_{d/2}} u \ge \sup_{\Omega_{d/2}} u.
\]
Therefore, $\inf_{\Omega_{d/2}} u\ge \dfrac{\eps}{2C}$. The claim then follows from $\inf_{\Omega} u\ge \min\{\inf_{\Omega_{d/2}} u, \inf_{\Omega \backslash\Omega_d} u \} \ge \dfrac{\eps}{2} \min\{1/C, 1\}\equiv \eps'$.
\end{proof}

We conclude this section by the following result on the differentiability of the datum $H$ with respect to the coefficients in the diffusion equation. This result justifies the linearization that we perform in Section~\ref{SEC:MultiPara}.
\begin{proposition}\label{THM:Differentiability}
The datum $H$ defined in~\eqref{EQ:Data Q TP} generated from an illumination $g\ge 0$ on $\partial \Omega$, viewed as the map
\begin{equation}
H[\gamma, \sigma,\mu]:
\begin{matrix}
(\gamma, \sigma,\mu) & \mapsto & \Gamma (\sigma u + \mu |u|u )\\
W^{1,2}(\Omega) \times L^{\infty}(\Omega)\times L^{\infty}(\Omega) & \to & W^{1,2}(\Omega)
\end{matrix} 
\end{equation}
is Fr\'{e}chet differentiable when the coefficients satisfies~\eqref{EQ:Coeff Assump A} and ~\eqref{EQ:Coeff Assump B}. The derivative at $(\gamma, \sigma, \mu)$ in the direction $(\delta\gamma, \delta\sigma, \delta\mu)\in W^{1,2}(\Omega) \times L^{\infty}(\Omega)\times L^{\infty}(\Omega)$ is given by
\begin{equation}
\begin{pmatrix}
H_\gamma'[\gamma,\sigma,\mu](\delta \gamma)\\
H_\sigma'[\gamma,\sigma,\mu](\delta \sigma)\\
H_\mu'[\gamma,\sigma,\mu](\delta \mu) 
\end{pmatrix} 
=\Gamma
\begin{pmatrix} 
\sigma v_1+ 2\mu u v_1\\
\delta \sigma u + 2\mu|u| v_2 \\
\delta \sigma v_3+ 2\mu |u| v_3+ \delta\mu |u|u 
\end{pmatrix},
\end{equation}
where $v_j$ ($1\le j\le 3$) is the solution to the diffusion equation
\begin{equation}\label{EQ:Diff Deri}
-\nabla\cdot(\gamma \nabla v_j) + (\sigma + 2\mu |u|) v_j = S_j,\quad 
\mbox{in}\  \Omega, \qquad v_j = 0,\quad \mbox{on}\ \partial\Omega
\end{equation}
with
\[
	S_1=\nabla\cdot\delta\gamma\nabla u,\qquad S_2=-\delta \sigma u,\qquad S_3=-\delta \mu |u|u.
\]
\end{proposition}
\begin{proof}
We show here only that $u$ is Fr\'{e}chet differentiable with respect to $\gamma$, $\sigma$ and $\mu$. The rest of the result follows from the chain rule. 

Let $(\delta\gamma, \delta\sigma, \delta\mu)\in W^{1,2}(\Omega) \times L^{\infty}(\Omega)\times L^{\infty}(\Omega)$ be such that $(\gamma', \sigma', \mu')=(\gamma+\delta\gamma, \sigma+\delta\sigma, \mu+\delta\mu)$ satisfies the bounds in~\eqref{EQ:Coeff Assump A}. Let $u'$ be the solution to ~\eqref{EQ:Diff TP} with coefficients $(\gamma', \sigma', \mu')$, and define $\wt u =u'-u$. We then verify that $\wt u$ solves the following linear diffusion equation
\begin{equation*}
	\begin{array}{rcll}
	-\nabla\cdot(\gamma \nabla \wt u) + \big[\sigma + \mu (u+u')\big] \wt u &=& \nabla\cdot \delta \gamma\nabla u' - \delta\sigma u' - \delta \mu u'^2, & \mbox{in}\ \Omega\\
\wt u &=& 0, & \text{on}\ \partial\Omega
	\end{array}
\end{equation*}
where we have used the fact that $u\ge 0$ and $u'\ge 0$ following Proposition~\ref{PROP:Maximum} (since $g\ge 0$ on $\partial\Omega$). Note also that both $u$ and $u'$ are bounded from above by Proposition~\ref{PROP:Maximum}. Therefore, $\sigma + \mu (u+u')$ is bounded from above. Therefore, we have the following standard estimate~\cite{GiTr-Book00}
\begin{multline}\label{EQ:Bound1}
	\|\wt u\|_{W^{1,2}(\Omega)}\le \mathfrak C_1 \left(\|\delta\gamma\nabla u'\|_{L^2(\Omega)}+\|\delta\sigma u'\|_{L^2(\Omega)}+\|\delta\mu u'^2\|_{L^2(\Omega)}\right) \\ \le \mathfrak C_1'(\|\delta\gamma\|_{L^\infty(\Omega)}+\|\delta\sigma \|_{L^\infty(\Omega)}+\|\delta\mu\|_{L^\infty(\Omega)}).
\end{multline}

Let $\wt{\wt u} =u'-u-(v_1+v_2+v_3)$ with $v_1$, $v_2$ and $v_3$ solutions to~\eqref{EQ:Diff Deri}. Then we verify that $\wt{\wt u}$ satisfies the equation
\begin{equation*}
	\begin{array}{rcll}
	-\nabla \cdot (\gamma\nabla \wt{\wt{u}}) + \big[\sigma + 2\mu u\big] \wt{\wt{u}} &=& \nabla\cdot \delta\gamma \nabla \wt u - \delta\sigma \wt u -\delta\mu (u'+u) \wt u, & \mbox{in}\ \Omega\\
\wt{\wt u} &=& 0, & \text{on}\ \partial\Omega
	\end{array}
\end{equation*}
Therefore, we have the following standard estimate
\begin{multline}\label{EQ:Bound2}
	\|\wt{\wt u}\|_{W^{1,2}(\Omega)}\le \mathfrak C_2 \left(\|\delta\gamma\nabla \wt u\|_{L^2(\Omega)}+\|\delta\sigma \wt u\|_{L^2(\Omega)}+\|\delta\mu \wt u^2\|_{L^2(\Omega)}\right)\\
\le \mathfrak C_2' \left(\|\delta\gamma\|_{L^\infty(\Omega)}\|\nabla \wt u\|_{L^2(\Omega)}+\|\delta\sigma\|_{L^\infty(\Omega)} \| \wt u\|_{L^2(\Omega)}+\|\delta\mu\|_{L^\infty(\Omega)} \|\wt u\|_{L^2(\Omega)}\right).
\end{multline}

We can thus combine~\eqref{EQ:Bound1} with~\eqref{EQ:Bound2} to obtain the bound
\[
	\|\wt{\wt u}\|_{W^{1,2}(\Omega)}\le \mathfrak C \left(\|\delta\gamma\|_{L^\infty(\Omega)}^2+\|\delta\sigma\|_{L^\infty(\Omega)}^2+\|\delta\mu\|_{L^\infty(\Omega)}^2\right).
\]
This concludes the proof.
\end{proof}
We observe from the above proof that differentiability of $H$ with respect to $\sigma$ and $\mu$ can be proven when viewed as a map $L^{\infty}(\Omega)\times L^{\infty}(\Omega) \to L^{\infty}(\Omega)$, following the maximum principles for solutions $\wt u$ and $\wt{\wt u}$. The same thing can not be done with respect to $\gamma$ since we can not control the term $\|\nabla\cdot\delta\gamma\nabla u'\|_{L^\infty(\Omega)}$ with $\|\delta\gamma\|_{L^\infty(\Omega)}$ without much more restrictive assumptions on $\delta\gamma$.

\section{Reconstructing absorption coefficients}
\label{SEC:Absorption}
We now study inverse problems related to the semilinear diffusion model~\eqref{EQ:Diff TP}. We first consider the case of reconstructing the absorption coefficients, assuming that the Gr\"uneisen coefficient $\Gamma$ and the diffusion coefficient $\gamma$ are both \emph{known}.

\subsection{One coefficient with single datum}
\label{SUBSEC:Single Coeff}

We now show that with one datum set, we can uniquely recover one of the two absorption coefficients.
\begin{proposition}\label{PROP:Stab Single Coeff}
Let $\Gamma$ and $\gamma$ be given. Assume that $g\ge \eps>0$ for some $\eps$. Let $H$ and $\wt H$ be the data sets corresponding to the coefficients $(\sigma, \mu)$ and $(\wt\sigma, \wt \mu)$ respectively. Then $H=\wt H$ implies $(u, \sigma+\mu |u|)=(\wt u, \wt\sigma+\wt\mu |\wt u|)$ provided that all coefficients satisfy~\eqref{EQ:Coeff Assump A}. Moreover, we have
\begin{equation}\label{EQ:Stab Single Coeff}
\|(\sigma+\mu |u|)-(\wt\sigma+\wt\mu \wt |u|)\|_{L^\infty(\Omega)}\le C\|H-\wt H\|_{L^\infty(\Omega)},
\end{equation}
for some constant $C$.
\end{proposition}
\begin{proof}
The proof is straightforward. Let $w=u-\wt u$. We check that $w$ solves
\begin{equation}\label{EQ:Diff Stab}
	-\nabla\cdot(\gamma\nabla w) = -\frac{1}{\Gamma}(H-\wt H), \quad\ \mbox{in}\ \Omega, \qquad w=0,\quad \mbox{on}\ \partial\Omega.
\end{equation}
Therefore $H=\wt H$ implies $w=0$ which is simply $u=\wt u$. This in turn implies that $\dfrac{H}{u}=\dfrac{\wt H}{\wt u}$, that is $\sigma+\mu |u|=\wt\sigma+\wt\mu |\wt u|$. Note that the condition $g\ge \eps>0$ implies that $u, \wt u\ge \eps'>0$ for some $\eps'$ following Theorem~\ref{THM:Positivity}. This makes it safe to take the ratios $H/u$ and $\wt H/\wt u$, and to omit the absolute values on $u$ and $\wt u$.

To derive the stability estimate, we first observe that
\begin{equation*}
|(\sigma+\mu |u|)-(\wt\sigma+\wt\mu |\wt u|)|=\dfrac{1}{\Gamma}|\dfrac{H}{u}-\dfrac{\wt H}{\wt u}|=|\dfrac{H(\wt u-u)+(H-\wt H) u}{\Gamma u\wt u}|.
\end{equation*}
Using the fact that $u$ and $\wt u$ are both bounded away from zero, and the triangle inequality, we have, for some constants $c_1$ and $c_2$,
\begin{equation}\label{EQ:Stab 1}
\|(\sigma+\mu |u|)-(\wt\sigma+\wt\mu |\wt u|)\|_{L^\infty(\Omega)}\le c_1\|\wt u-u\|_{L^\infty(\Omega)}+c_2\|H-\wt H\|_{L^\infty(\Omega)}.
\end{equation}
On the other hand, classical theory of elliptic equations allows us to derive, from~\eqref{EQ:Diff Stab}, the following bound, for some constant $c_3$,
\begin{equation}\label{EQ:Stab 2}
\|u-\wt{u}\|_{L^\infty(\Omega)} \leq c_3 \|H-\wt{H}\|_{L^{\infty}(\Omega)}.
\end{equation}
The bound in~\eqref{EQ:Stab Single Coeff} then follows by combining~\eqref{EQ:Stab 1} and~\eqref{EQ:Stab 2}.
\end{proof}

The above proof provides an explicit algorithm to reconstruct one of $\sigma$ and $\mu$ from one datum. Here is the procedure. We first solve
\begin{equation}\label{EQ:Diff Rec Single}
	-\nabla\cdot(\gamma\nabla u) = -\frac{1}{\Gamma}H, \quad\ \mbox{in}\ \Omega, \qquad u=g,\quad \mbox{on}\ \partial\Omega
\end{equation}
for $u$ since $\Gamma$ and $\gamma$ are known. We then reconstruct $\sigma$ as
\begin{equation}\label{EQ:Sigma Expl}
	\sigma = \frac{H}{\Gamma u} - \mu |u|,
\end{equation}
if $\mu$ is known, or reconstruct $\mu$ as
\begin{equation}\label{EQ:Mu Expl}
\mu = \frac{H}{\Gamma u|u|} - \frac{\sigma}{|u|},
\end{equation}
if $\sigma$ is known.

The stability estimate~\eqref{EQ:Stab Single Coeff} can be made more explicit when one of the coefficients involved is known. For instance, if $\mu$ is known, then we have  
\begin{equation*}
|\sigma-\wt \sigma| = \dfrac{1}{\Gamma}\big|\dfrac{H}{u}-\mu |u| -\big(\dfrac{\wt H}{\wt u} - \mu |\wt u|\big)\big|=\dfrac{1}{\Gamma}\big|\dfrac{(H(\wt u-u)+(H-\wt H)u}{u\wt u}-\mu(|u|-|\wt u|)\big|.
\end{equation*}
This leads to, using the triangle inequality again,
\begin{equation*}
\|\sigma-\wt\sigma\|_{L^\infty(\Omega)}\le c_1'\|\wt u-u\|_{L^\infty(\Omega)}+c_2'\|H-\wt H\|_{L^\infty(\Omega)}.
\end{equation*}
Combining this bound with~\eqref{EQ:Stab 2}, we have
\begin{equation}
\|\sigma-\wt \sigma\|_{L^{\infty}(\Omega)}\le C' \| H -\wt H\|_{L^{\infty}(\Omega)},
\end{equation}
for some constant $C'$. In the same manner, we can derive 
\begin{equation}
\|\mu-\wt \mu\|_{L^{\infty}(\Omega)}\le C'' \| H -\wt H\|_{L^{\infty}(\Omega)},
\end{equation}
for the reconstruction of $\mu$ if $\sigma$ is known in advance.

\subsection{Two coefficients with two data sets}
\label{SUBSEC:Two Coeff}

We see from the previous result that we can reconstruct $\sigma+\mu |u|$ when we have one datum. If we have data generated from two different sources $g_1$ and $g_2$, then we can reconstruct $\sigma+\mu |u_1|$ and $\sigma+\mu |u_2|$ where $u_1$ and $u_2$ are the solutions to the diffusion equation~\eqref{EQ:Diff TP} corresponding to $g_1$ and $g_2$ respectively. If we can choose $g_1$ and $g_2$ such that $|u_2|-|u_1| \neq 0$ almost everywhere, we can uniquely reconstruct the pair $(\sigma, \mu)$. This is the idea we have in the following result.
\begin{proposition}\label{PROP:Stab Two Coeff}
Let $\Gamma$ and $\gamma$ be given. Let $(H_1, H_2)$ and $(\wt H_1, \wt H_2)$ be the data sets corresponding to the coefficients $(\sigma, \mu)$ and $(\wt\sigma, \wt \mu)$ respectively that are generated with the pair of sources $(g_1, g_2)$. Assume that $g_i\ge \eps>0$, $i=1,2$, and $g_1-g_2\ge \eps'>0$ for some $\eps$ and $\eps'$. Then $(H_1,H_2)=(\wt H_1, \wt H_2)$ implies $(\sigma, \mu)=(\wt\sigma, \wt\mu)$ provided that all coefficients involved satisfy~\eqref{EQ:Coeff Assump A}. Moreover, we have
\begin{equation}\label{EQ:Stab Two Coeff}
\|\sigma-\wt\sigma\|_{L^\infty(\Omega)}+\|\mu-\wt\mu\|_{L^\infty(\Omega)} \le \wt C\left(\|H_1-\wt H_1\|_{L^\infty(\Omega)}+\|H_2-\wt H_2\|_{L^\infty(\Omega)}\right),
\end{equation}
for some constant $\wt C$.
\end{proposition}
\begin{proof}
Let $w_i=u_i-\wt u_i$, $i=1,2$. Then $w_i$ solves
\begin{equation}\label{EQ:Diff Stab i}
	-\nabla\cdot(\gamma\nabla w_i) = -\frac{1}{\Gamma}(H_i-\wt H_i), \quad\ \mbox{in}\ \Omega, \qquad w_i=0,\quad \mbox{on}\ \partial\Omega.
\end{equation}
Therefore $H_i=\wt H_i$ implies $u_i=\wt u_i$ and 
\[
	\sigma+\mu |u_i|=\wt\sigma+\wt\mu |u_i| .
\]
Collecting the results for both data sets, we have
\begin{equation}\label{EQ:Unique Two Coeff}
	\left(
	\begin{array}{cc}
	1 & |u_1|\\
	1 & |u_2|
	\end{array}
	\right)
	\left(
	\begin{array}{c}
	\sigma\\
	\mu
	\end{array}
	\right)
	=
	\left(
	\begin{array}{cc}
	1 & |u_1|\\
	1 & |u_2|
	\end{array}
	\right)
	\left(
	\begin{array}{cc}
	\wt\sigma\\
	\wt\mu
	\end{array}
	\right).
\end{equation}
When $g_1$ and $g_2$ satisfy the requirements stated in the proposition, we have $u_1-u_2\ge \eps'>0$ for some $\eps'$. Therefore, the matrix $\left(
	\begin{array}{cc}
	1 & |u_1|\\
	1 & |u_2|
	\end{array}
	\right)$
is invertible. We can then remove this matrix in~\eqref{EQ:Unique Two Coeff} to show that $(\sigma, \mu)=(\wt\sigma, \wt\mu)$.

To get the stability estimate in~\eqref{EQ:Stab Two Coeff}, we first verify that
\begin{equation*}
(\sigma-\wt\sigma)+(\mu-\wt\mu) |u_i|=\dfrac{H_i}{u_i}-\dfrac{\wt H_i}{\wt u_i}-\wt\mu(|u_i|-|\wt u_i|),\qquad i=1,2.
\end{equation*}
This leads to,
\begin{equation*}
	\left(
	\begin{array}{cc}
	1 & |u_1|\\
	1 & |u_2|
	\end{array}
	\right)
	\left(
	\begin{array}{c}
	\sigma-\wt\sigma\\
	\mu-\wt\mu
	\end{array}
	\right)
	=
	\left(
	\begin{array}{cc}
	\dfrac{H_1}{u_1}-\dfrac{\wt H_1}{\wt u_1}-\wt\mu(|u_1|-|\wt u_1|)\\
	\dfrac{H_2}{u_2}-\dfrac{\wt H_2}{\wt u_2}-\wt\mu(|u_2|-|\wt u_2|)
	\end{array}
	\right).
\end{equation*}
Therefore, we have
\begin{equation*}
	\left(
	\begin{array}{c}
	\sigma-\wt\sigma\\
	\mu-\wt\mu
	\end{array}
	\right)
	=
	\left(
	\begin{array}{cc}
	1 & |u_1|\\
	1 & |u_2|
	\end{array}
	\right)^{-1}
	\left(
	\begin{array}{cc}
	\dfrac{H_1(\wt u_1-u_1)+(H_1-\wt H_1) u_1}{u_1\wt u_1}-\wt\mu(|u_1|-|\wt u_1|)\\
	\dfrac{H_2(\wt u_2-u_2)+(H_2-\wt H_2) u_2}{u_2\wt u_2}-\wt\mu(|u_2|-|\wt u_2|)
	\end{array}
	\right).
\end{equation*}
It then follows that
\begin{multline}\label{EQ:Stab 1 i}
\|\sigma-\wt\sigma\|_{L^\infty(\Omega)}+\|\mu-\wt\mu\|_{L^\infty(\Omega)} \\ 
\le c\left(\|H_1-\wt H_1\|_{L^\infty(\Omega)}+\|H_2-\wt H_2\|_{L^\infty(\Omega)}+\|u_1-\wt u_1\|_{L^\infty(\Omega)}+\|u_2-\wt u_2\|_{L^\infty(\Omega)}\right).
\end{multline}
Meanwhile, we have, from~\eqref{EQ:Diff Stab i},
\begin{equation}\label{EQ:Stab 2 i}
	\|u_i-\wt{u}_i\|_{L^\infty(\Omega)} \leq c' \|H_i-\wt H_i\|_{L^{\infty}(\Omega)},\quad i=1,2.
\end{equation}
The bound in~\eqref{EQ:Stab Two Coeff} then follows from~\eqref{EQ:Stab 1 i} and~\eqref{EQ:Stab 2 i}.
\end{proof}

\section{Reconstructing absorption and diffusion coefficients}
\label{SEC:MultiPara}

We now study inverse problems where we intend to reconstruct more than the absorption coefficients. We start with a non-uniqueness result on the simultaneous reconstructions of all four coefficients $\Gamma$, $\gamma$, $\sigma$, and $\mu$.

\subsection{Non-uniqueness in reconstructing $(\Gamma, \gamma, \sigma, \mu)$}

Let us assume for the moment that $\gamma^{1/2}\in \cC^2(\Omega)$. We introduce the following Liouville transform
\begin{equation}\label{EQ:Liouville}
	v=\sqrt{\gamma} u .
\end{equation}
We then verify that the semilinear diffusion equation~\eqref{EQ:Diff TP} is transformed into the following equation under the Liouville transform:
\begin{equation}\label{EQ:Diff TP Liouville}
	\Delta v-\left(\dfrac{\Delta\gamma^{1/2}}{\gamma^{1/2}}+\dfrac{\sigma}{\gamma}+\dfrac{\mu}{\gamma^{3/2}}|v|\right)v=0,\quad \mbox{in}\ \Omega, \qquad v=\gamma^{1/2} g,\quad \mbox{on}\ \partial\Omega
\end{equation}
and the datum $H$ is transformed into
\begin{equation}
	H(\bx) =\Gamma(\bx) \left(\dfrac{\sigma}{\gamma^{1/2}}v(\bx)+\dfrac{\mu}{\gamma} v^2(\bx)\right).
\end{equation}
Let us now define the following functionals:
\begin{equation}
	\alpha=\dfrac{\Delta\gamma^{1/2}}{\gamma^{1/2}}+\dfrac{\sigma}{\gamma}, \qquad \beta=\dfrac{\mu}{\gamma^{3/2}}, \qquad \zeta_1=\Gamma\dfrac{\sigma}{\gamma^{1/2}}, \qquad \zeta_2=\Gamma\dfrac{\mu}{\gamma}.
\end{equation}
The following result says that once $(\alpha, \beta, \zeta_1)$ or $(\alpha, \beta, \zeta_2)$ is known, introducing new data would not bring in new information.
\begin{theorem}
	Let $\gamma^{1/2}|_{\partial\Omega}$ be given and assume that $\gamma^{1/2}\in \cC^2(\Omega)$. Assume that either $(\alpha, \beta, \zeta_1)$ or $(\alpha, \beta, \zeta_2)$ is known, and $H$ is among the data used to determine them. Then for any given new illumination $\wt g$, the corresponding datum $\wt H$ is uniquely determined by $(\wt g, H)$.
\end{theorem}
\begin{proof}
Let us first rewrite the datum as $H=\zeta_1 v+ \zeta_2 v^2$. When $\alpha$ and $\beta$ are known, we know the solution $v$ of ~\eqref{EQ:Diff TP Liouville} for any given $g$. If $\zeta_1$ is also known, we know also $\zeta_1 v$. We therefore can form the ratio
\[
	\dfrac{\wt H-\zeta_1 \wt v}{H-\zeta_1 v} = \dfrac{\zeta_2 \wt v^2}{\zeta_2 v^2}=\dfrac{\wt v^2}{v^2}.
\]
We then find $\wt H$ as $\wt H=\dfrac{\wt v^2}{v^2}(H-\zeta_1 v)+\zeta_1\wt v$. If $\zeta_1$ is not known but $\zeta_2$ is known, we can form the ratio
\[
	\dfrac{\wt H-\zeta_2 \wt v^2}{H-\zeta_2 v^2} = \dfrac{\zeta_1 \wt v}{\zeta_1 v}=\dfrac{\wt v}{v}.
\]
This gives $\wt H=\dfrac{\wt v}{v}(H-\zeta_2 v^2)+\zeta_2\wt v^2$. The proof is complete.
\end{proof}

The above theorem says that we can at most reconstruct the triplet $(\alpha, \beta, \zeta_1)$ or the triplet $(\alpha, \beta, \zeta_2)$. Neither triplet would allow the unique determination of the four coefficients $(\Gamma, \gamma, \sigma, \mu)$. Once one of the triplets is determined, adding more data is not helpful in terms of uniqueness of reconstructions. 

Similar non-uniqueness results were proved in the case of the regular PAT~\cite{BaRe-IP11,BaRe-CM11}. In that case, it was also shown that if the Gr\"uneisen coefficient $\Gamma$ is known, for instance from multi-spectral measurements~\cite{BaRe-IP12,MaRe-CMS14}, one can uniquely reconstruct the absorption coefficient and the diffusion coefficient simultaneously. In the rest of this section, we consider this case, that is, $\Gamma$ is known, for our TP-PAT model.

\subsection{Linearized reconstruction of $(\gamma, \sigma, \mu)$}

We study the problem of reconstructing $(\gamma, \sigma, \mu)$, assuming $\Gamma$ is known, in linearized setting following the general theory of overdetermined elliptic systems developed in~\cite{DoNi-CPAM55,Solonnikov-JSM73}. For the sake of the readability of the presentation below, we collect some necessary terminologies in the theory of overdetermined elliptic systems in Appendix A. We refer interested readers to~\cite{Bal-CM13,KuSt-IP12,WiSc-IP15} for overviews of the theory in the context of hybrid inverse problems and references therein for more technical details on the theory. Our presentation below follows mainly~\cite{Bal-CM13}.

We linearize the nonlinear inverse problem around background coefficients $(\gamma, \sigma, \mu)$, assuming that we have access to data collected from $J$ different illumination sources $\{g_j\}_{j=1}^J$. We denote by $(\delta\gamma, \delta\sigma, \delta\mu)$ the perturbations to the coefficients. Let $u_j$ be the solution to~\eqref{EQ:Diff TP} with source $g_j$ and the background coefficients. We then denote by $\delta u_j$ the perturbation to solution $u_j$. Following the calculations in Proposition~\ref{THM:Differentiability}, we have, for $1\leq j \leq J$, 
\begin{align}
\label{EQ:GSM Lin Sys EQ} 
-\nabla\cdot (\delta \gamma\nabla u_j) - \nabla\cdot(\gamma\nabla \delta u_j) + \delta \sigma u_j + \delta \mu |u_j|u_j + (\sigma+2\mu |u_j|) \delta u_j &= 0, & \mbox{in}\ \Omega\\
\label{EQ:GSM Lin Sys DA} 
\delta \sigma u_j + \delta \mu |u_j|u_j + (\sigma+2\mu |u_j|) \delta u_j &= \delta H_j/\Gamma, & \mbox{in}\ \Omega
\end{align}
To simplify our analysis, we rewrite the above system into, $1\leq j \leq J$,  
\begin{eqnarray}
\label{EQ:GSM Lin Sys EQ 2} 
	-\nabla\cdot (\delta \gamma \nabla u_j) - \nabla\cdot(\gamma \nabla \delta u_j)  &= -\delta H_j/\Gamma, & \mbox{in}\ \Omega\\
 \label{EQ:GSM Lin Sys DA 2}  u_j \delta \sigma +  |u_j|u_j \delta \mu + (\sigma + 2\mu |u_j|) \delta u_j &= + \delta H_j/\Gamma, & \mbox{in}\ \Omega
\end{eqnarray}
This is a system of $2J$ differential equations for $J+3$ unknowns $\{\delta \gamma, \delta \sigma, \delta \mu, \delta u_1, \ldots, \delta u_J\}$. The system is formally overdeterminted when $J>3$.

To supplement the above system with appropriate boundary conditions, we first observe that the boundary conditions for the solutions $\{\delta u_j\}_{j=1}^J$ are given already. They are homogeneous Dirichlet conditions since $g$ does not change when the coefficients change. The boundary conditions for $(\delta\gamma, \delta\sigma, \delta\mu)$ are what need to be determined. In the case of single-photon PAT, it has been shown that one needs to have $\gamma_{|\partial\Omega}$ \emph{known} to have uniqueness in the reconstruction~\cite{BaRe-IP11,ReGaZh-SIAM13}. This is also expected in our case. We therefore take $\delta\gamma_{|\partial\Omega}=\phi_1$ for some known $\phi_1$. The boundary conditions for $\sigma$ and $\mu$ are given by the data. In fact, on the boundary, $u=g$. Therefore, we have, from~\eqref{EQ:GSM Lin Sys DA 2} which holds on $\partial\Omega$, that
\[
	g_j \delta \sigma + |g_j|g_j\delta \mu = \delta H_j/\Gamma,\qquad \mbox{on}\ \partial\Omega.
\]
If we have two perturbed data sets $\{\delta H_1, \delta H_2\}$ with $g_1$ and $g_2$ sufficiently different, we can then uniquely reconstruct $(\delta\sigma_{|\partial\Omega}, \delta\mu_{|\partial\Omega})$:
\[
	\delta \sigma_{|\partial\Omega}=\frac{\delta H_1 |g_2|g_2-\delta H_2 |g_1|g_1}{\Gamma g_1 g_2(|g_2|-|g_1|)}\equiv\phi_2,\qquad \delta\mu_{|\partial\Omega}=\frac{\delta H_2g_1-\delta H_1g_2}{\Gamma g_1g_2(|g_2|-|g_1|)}\equiv \phi_3.
\]
Therefore, we have the following Dirichlet boundary condition for the unknowns
\begin{equation}\label{EQ:BC}
	(\delta \gamma, \delta \sigma, \delta \mu, \delta u_1, \ldots, \delta u_J)=(\phi_1, \phi_2, \phi_3, 0, \cdots, 0) .
\end{equation}

Let us introduce $v=(\delta \gamma, \delta \sigma, \delta \mu, \delta u_1, \ldots, \delta u_J)$, $\mathcal{S}=(-\delta H_1, \delta H_1, \ldots,-\delta H_J, \delta H_J)/\Gamma$, and $\phi=(\phi_1,\phi_2,\phi_3, 0, \cdots, 0)$. We can then write the linearized system of equations~\eqref{EQ:GSM Lin Sys EQ 2}-\eqref{EQ:GSM Lin Sys DA 2} and the corresponding boundary conditions into the form of
\begin{equation}\label{EQ:Elliptic Sys Form}
	\cA(\bx,D) v = \cS, \quad \text{in}\ \Omega \qquad 
	\cB(\bx,D)v = \phi, \quad \text{on}\ \partial \Omega
\end{equation}
where $\cA$ is a matrix differential operator of size $M\times N$, $M=2J$ and $N=3+J$, while $\cB$ is the identity operator. The symbol of $\cA$ is given as
\begin{equation}\label{EQ:A gsm}
\cA(\bx,\fki\bxi) = 
\begin{pmatrix}
-\fki \bV_1\cdot \bxi - \Delta u_1 & 0 & 0 & \gamma |\bxi|^2-\fki\bxi\cdot\nabla\gamma  & \ldots & 0 \\
0 & u_1 & |u_1|u_1 & \sigma + 2\mu |u_1|  & \ldots & 0 \\
 \vdots & \vdots & \vdots & \vdots & \vdots & \vdots\\
 -\fki\bV_J\cdot\bxi -\Delta u_J & 0 & 0 & 0 & \ldots & \gamma |\bxi|^2-\fki\bxi\cdot\nabla\gamma\\
0  & u_J & |u_J|u_J & 0 & \ldots & \sigma + 2\mu |u_J|
\end{pmatrix}
,
\end{equation}
with $\bV_j = \nabla u_j,\ 1\le j\le J$ and $\bxi\in\bbS^{n-1}$ ($\bbS^{n-1}$ being the unit sphere in $\bbR^n$).

It is straightforward to check that if we take the associated Douglis-Nirenberg numbers as 
\begin{equation}\label{EQ:D-N gsm}
	\{s_i\}_{i=1}^{2J} = (0,-2,\ldots,0,-2), \qquad \{t_j\}_{j=1}^{J+3}=(1,2,2,2,\ldots,2),
\end{equation}
the principal part of $\cA$ is simply $\cA$ itself with the $-\fki\bxi\cdot\nabla\gamma$ and $-\Delta u_j$ ($1\le j\le J$) terms removed.

In three-dimensional case, we can establish the following result.
\begin{theorem}\label{THM:Ellipticity}
Let $n=3$. Assume that the background coefficients $\gamma \in \cC^4(\Omega)$, $\sigma \in \cC^2(\Omega)$, and $\mu \in \cC^1(\Omega)$ satisfy the bounds in~\eqref{EQ:Coeff Assump A}. Then, there exists a set of $J\ge n+1$ illuminations $\{g_j\}_{j=1}^J$ such that $\cA$ is elliptic. Moreover, the corresponding elliptic system $(\cA, \cB)$, with boundary condition~\eqref{EQ:BC}, satisfies the Lopatinskii criterion.
\end{theorem}
\begin{proof}
Let us first rewrite the principal symbol $\cA_0$ as
\begin{equation*}
	\cA_{0}(\bx,\mathfrak{i}\bxi) = 
\begin{pmatrix}
-\fki \bV_1\cdot \bxi & 0 & 0 & \gamma|\bxi|^2  & \ldots & 0 \\
\fki \frac{\bV_1\cdot \bxi}{\gamma|\bxi|^2} (\sigma + 2\mu|u_1|)u_1 & u_1 & |u_1|u_1 & 0  & \ldots & 0 \\
 \vdots & \vdots & \vdots & \vdots & \vdots & \vdots\\
 -\fki \bV_J\cdot \bxi & 0 & 0 & 0 & \ldots & \gamma|\bxi|^2\\
\fki \frac{\bV_J\cdot \bxi}{\gamma|\bxi|^2} (\sigma + 2\mu|u_J|)u_J  & u_J & |u_J|u_J & 0 & \ldots & 0
\end{pmatrix}
.
\end{equation*}
It is then straightforward to check that $\cA_{0}(\bx,\fki\bxi)$ is of full-rank as long as the following sub-matrix is of full-rank:
\begin{equation*}
\wt{\cA}_{0}(\bx,\fki\bxi) = 
\begin{pmatrix}
	\fki \frac{\bV_1\cdot \bxi}{\gamma |\bxi|^2} (\sigma + 2\mu|u_1|)u_1 & u_1 & |u_1|u_1 \\
 \vdots & \vdots & \vdots \\
\fki \frac{\bV_J\cdot \bxi}{\gamma |\bxi|^2}(\sigma + 2\mu|u_J|)u_J  & u_J & |u_J|u_J 
\end{pmatrix}
.
\end{equation*}

To simplify the calculation, we introduce $\Sigma_j = \sigma + 2\mu |u_j|$, $\widehat{F}_j = \bV_j\cdot \bxi$. We also eliminate the non-zero common factor $\dfrac{\fki}{\gamma|\bxi|^2}$ from the first column and $u_j$ from each row. Without loss of generality, we check the determinant of first $3$ (since $J\ge n+1= 4$) rows of the simplified version of the submatrix $\tilde{\cA}_{0}(\bx,\fki\bxi)$. This determinant is given as
\begin{multline*}
\det(\wt \cA_0) = \widehat{F}_1\frac{\Sigma_{1}}{u_1}(|u_3| - |u_2|) + \widehat{F}_2\frac{\Sigma_{2}}{u_2}(|u_1| - |u_3|) + \widehat{F}_3\frac{\Sigma_{3}}{u_3}(|u_2| - |u_1|)\\
= \frac{\Sigma_{1}\Sigma_{2}\Sigma_{3}}{u_1 u_2 u_3}\left(\widehat{F}_1\frac{u_3 u_2(|u_3| - |u_2|)}{\Sigma_{3} \Sigma_{2} } + \widehat{F}_2\frac{u_1 u_3(|u_1| - |u_3|)}{\Sigma_{1} \Sigma_{3} } + \widehat{F}_3\frac{u_2 u_1(|u_2| - |u_1|)}{\Sigma_{2} \Sigma_{1} }\right).
\end{multline*}

With the assumptions on the background coefficients, we can take $u_j$ to be the complex geometric optics solution constructed following Theorem~\ref{THM:CGO} (in Appendix B) for $\brho_j$ with boundary condition $g_j$. Then we have
\begin{multline*}
\widehat{F}_k\frac{u_i u_j(|u_i| - |u_j|)}{\Sigma_{i} \Sigma_{j} } = \frac{u_i u_j (|u_i|-|u_j|)}{\Sigma_i \Sigma_j} \nabla u_k \cdot \bxi
 = u_i u_j u_k \frac{(|u_i|-|u_j|)}{\Sigma_i \Sigma_j}(\brho_k + O(1))\cdot \bxi .
\end{multline*}
This gives us,
\begin{equation}\label{EQ:DetA}
\det(\wt \cA_0) \sim \Bigl( \Sigma_1(|u_2|-|u_3|)\brho_1 + \Sigma_2(|u_3|-|u_1|)\brho_2 + \Sigma_3(|u_1|-|u_2|)\brho_3 \Bigr) \cdot \bxi .
\end{equation}
Let us define $\alpha_k = \Sigma_k(|u_i|-|u_j|)$ with $(k, i, j)\in\{(1, 2, 3), (2, 3, 1), (3, 1, 2)\}$. Then we have $\alpha_1+\alpha_2+\alpha_3 = 0$. 
Let $(\be_1, \be_2, \be_3)$ an orthonormal basis for $\bbR^3$. Then $\bxi=\sum_{k=1}^3 c_k \be_k$ with $|c_1|^2+|c_2|^2+|c_3|^2 = 1$. We take
\begin{align*}
\brho_1 &= \beta_1 \big(\tau_1 \be_1 + \fki \wt \tau_1\be_2\big),\\
\brho_2 &= \beta_2 \big(\tau_2 \be_2 + \fki \wt \tau_2\be_3\big),\\
\brho_3 &= \beta_3 \big(\tau_3 \be_3 + \fki \wt \tau_3\be_1\big),
\end{align*}
where $|\tau_k| =|\wt\tau_k|$, $\forall 1\le k\le 3$. It is straightforward to verify that $\brho_k\cdot\brho_k=0$, $|\brho_k|=\sqrt{2}|\tau_k||\beta_k|$ for all $1\le k\le 3$. We now deduce from~\eqref{EQ:DetA} that $\det(\wt \cA_0) \sim \square_\cR+\fki \square_\cI$ where
\begin{equation*}
	\square_\cR = \alpha_1\tau_1\beta_1 c_1+\alpha_2\tau_2\beta_2 c_2+\alpha_3\tau_3\beta_3 c_3, \qquad 
	\square_\cI = \alpha_1\wt\tau_1\beta_1 c_2+\alpha_2\wt\tau_2 c_3+\alpha_3\wt\tau_3 \beta_3c_1.
\end{equation*}
Take $\beta_1=\beta_2=\beta_3$, $\tau_k=\wt\tau_k=1$, $1\le k\le 3$. Then $\det(\cA_0)\ne 0$ unless $c_1=c_2=c_3$. [ This is because if $\det(\cA_0)=0$, we have $\square_\cR=0$, $\square_\cI=0$, and $a_1+a_2+a_3=0$. That is
\begin{equation*}
\begin{pmatrix}
	1 & 1 & 1\\
	c_1 & c_2 & c_3\\
	c_2 & c_3 & c_1
\end{pmatrix}
\begin{pmatrix}
	\alpha_1\\
	\alpha_2\\
	\alpha_3
\end{pmatrix}
=
\begin{pmatrix}
	0\\
	0\\
	0
\end{pmatrix}
.
\end{equation*}
This contradicts the construction  of $\{\alpha_k\}_{k=1}^3$]. 
Let us now take
\[
	\brho_4=2\brho_3.
\]
Then the submatrix formed by $u_1$, $u_2$ and $u_4$ will have full rank when $c_1=c_2=c_3$. Therefore the submatrix formed by $u_1$, $u_2$, $u_3$ and $u_4$ is of full rank for any $\bxi$.

To show that $(\cA, \cB)$ satisfies the Lopatinskii criterion given in Definition~\ref{DEF:Lop BC} for a set of well chosen $u_j$, we first observe that since $\cB=\cI$, we have, from the definition in~\eqref{EQ:Eta}, 
\begin{equation}\label{EQ:Eta gsm}
	\{\eta_j\}_{j=1}^{J+3}=\{-1,\cdots,-1\}
\end{equation} 
with the selection of the Douglis-Nirenberg numbers $\{s_i\}_{i=1}^{2J}$ and $\{t_j\}_{j=1}^{J+3}$ in~\eqref{EQ:D-N gsm}, and the principal part of $\cB$ has components $\cB_{0,11}=1$ and $\cB_{0,k\ell}=0$ otherwise. Therefore, the system of differential equations in ~\eqref{EQ:Lop A} and~\eqref{EQ:Lop B} takes the following form
\begin{eqnarray}
\label{EQ:GSM Lin Sys EQ 2 BC} 
	(\bV_j\cdot\bzeta - \fki \bV_j\cdot\bnu\frac{d}{dz})\delta\gamma(z)-\gamma(-|\bzeta|^2 +\frac{d^2}{dz^2})\delta u_j(z) &= 0, & z> 0\\
 \label{EQ:GSM Lin Sys DA 2 BC}  u_j \delta \sigma(z) +  |u_j|u_j \delta \mu(z) + (\sigma + 2\mu |u_j|) \delta u_j(z) &= 0, & z>0\\
 \label{EQ:GSM Lin Sys DA 2 BC3}  \delta\gamma &= 0, & z=0
\end{eqnarray}
where $\gamma$, $\sigma$, $\mu$, $u_j$ and $\bV_j$ ($1\le j\le J$, are all evaluated at $\by\in\partial\Omega$. We first deduce from~\eqref{EQ:GSM Lin Sys DA 2 BC} that, for $1\le j\le J$, 
\begin{equation*}
	\delta u_j = -\frac{u_j}{\Sigma_j} \delta \sigma - \frac{|u_j|u_j}{\Sigma_j}\delta\mu, \quad z>0.
\end{equation*}
Plugging this into~\eqref{EQ:GSM Lin Sys EQ 2 BC}, we obtain that, for $1\leq j \leq J$,
\begin{equation*}
(\bV_j\cdot \bzeta - \fki\bV_j \cdot \bnu \frac{d}{dz}) \delta \gamma + \gamma (-|\bzeta|^2 + \frac{d^2}{dz^2})\left(\frac{u_j}{\Sigma_j}\delta \sigma + \frac{|u_j|u_j}{\Sigma_j}\delta \mu\right)=0, \quad z>0.
\end{equation*}
Without loss of generality, we consider the system formed by $u_1, u_2, u_3$. Let $\wt{F}_j = \bV_j \cdot \bnu$, $p_j=\frac{u_j}{\Sigma_j}$ and $q_j=\frac{u_j^2}{\Sigma_j}$. We look for eigenvalues of the system as the root of
\begin{equation*}
\det
\begin{pmatrix}
\wh{F}_1- \fki\lambda \wt{F}_1 & p_1 \gamma(\lambda^2 -|\bzeta|^2) & q_1 \gamma(\lambda^2 -|\bzeta|^2) \\
\wh{F}_2- \fki\lambda \wt{F}_2 & p_2 \gamma(\lambda^2 -|\bzeta|^2) & q_2 \gamma(\lambda^2 -|\bzeta|^2)\\
\wh{F}_3- \fki\lambda \wt{F}_3 & p_3 \gamma(\lambda^2 -|\bzeta|^2) & q_3 \gamma(\lambda^2 -|\bzeta|^2)
\end{pmatrix}
=0.
\end{equation*}
We observe first that the above equation admits two \emph{repeated} roots $\lambda_{2,3} = \pm |\bzeta|$. Besides that, we have another root
\begin{equation*}
\lambda_1 = -\fki\frac{\wh F_1(p_2q_3-p_3q_2)+\wh F_2(p_3q_1-p_1q_3)+\wh F_3(p_1q_2-p_2q_1)}{\wt F_1(p_2q_3-p_3q_2)+\wt{F}_2(p_3q_1-p_1q_3)+\wt{F}_3(p_1q_2-p_2q_1)}.
\end{equation*}
Moreover, the eigenvectors corresponding to $\lambda_{2,3}$ are of the form
\[
	\boldsymbol \pi_{2,3}=\begin{pmatrix}
	0\\
	x\\
	y
	\end{pmatrix}
\]
with $x$ and $y$ arbitrary. Therefore, $\delta\gamma(z)=ce^{i|\lambda_1|z}$. Using the boundary condition $\delta\gamma(0)=0$ and the decay condition $\delta\gamma(z)\to 0$ as $z\to\infty$, we conclude that $\delta\gamma(z)\equiv 0$. This in turn implies, from~\eqref{EQ:GSM Lin Sys DA 2 BC}, that $\delta\sigma(z)\equiv 0$ and $\delta\mu(z)\equiv 0$. The proof is complete.
\end{proof}

For the set of Douglis-Nirenberg numbers $\{s_i\}$ and $\{t_j\}$ in~\eqref{EQ:D-N gsm}, as well as the parameters $\{\eta_k\}$ given in~\eqref{EQ:Eta gsm}, we defined the function space, parameterized by $\ell>n+\frac{1}{2}$,
\begin{equation*}
\cW_\ell = W^{\ell-s_1,2}(\Omega) \times \ldots \times W^{\ell-s_{2J},2}(\Omega) \times W^{\ell-\eta_1-\frac{1}{2},2}(\partial \Omega)\times \ldots W^{\ell-\eta_{3}-\frac{1}{2},2}(\partial \Omega).
\end{equation*}
We have the following uniqueness and stability result.
\begin{theorem}
Under the same conditions of Theorem~\ref{THM:Ellipticity}, let $\{\delta H_j\}_{j=1}^J$ and $\{\wt{\delta H_j}\}_{j=1}^J$ be the data sets generated with $(\delta\gamma, \delta\sigma, \delta\mu)$ and $(\wt{\delta\gamma}, \wt{\delta\sigma}, \wt{\delta\mu})$ respectively. Assume that the data are such that $(\cS, \phi)\in\cW_\ell$ and $(\wt\cS, \wt\phi)\in\cW_\ell$. Then there exists a set of $J\ge n+1$ boundary illuminations, $\{g_j\}_{j=1}^{J}$, such that $\{\delta H_j\}_{j=1}^J=\{\wt{\delta H_j}\}_{j=1}^J$ (resp. $(\cS, \phi)=(\wt\cS, \wt\phi)$) implies  $(\delta\gamma, \delta\sigma, \delta\mu)=(\wt{\delta\gamma}, \wt{\delta\sigma}, \wt{\delta\mu})$ (resp. $v=\tilde v$) if $\delta\gamma_{|\partial\Omega}=\wt{\delta\gamma}_{|\partial\Omega}$. Moreover, the following stability estimate holds:
\begin{equation}\label{EQ:Stab Lin Sys gsm}
	\sum_{j=1}^{J+3}\Vert v_j-\wt v_j \Vert_{W^{\ell+t_j,2}(\Omega)} \leq C\Bigl(\sum_{i=1}^{2J}\Vert \cS_{i}-\wt\cS_{i}\Vert_{W^{\ell-s_i,2}(\Omega)}+\sum_{k=1}^{3} \Vert \phi_k-\wt\phi_k \Vert_{W^{\ell-\eta_k-\frac{1}{2},2}(\partial\Omega)}\Bigr),
\end{equation}
for all $\ell>n+\frac{1}{2}$.
\end{theorem}
\begin{proof}
	We start with the uniqueness result. Let $\delta H_j=0$, $1\le j\le 3$, we then have that
	\[
		u_j \delta\sigma + |u_j|u_j \delta \mu + (\sigma+2\mu|u_j|)\delta u_j=0,\quad 1\le j\le 3.\]
	We can eliminate the variables $\delta\sigma$ and $\delta\mu$ to have, with $\cE=\{(1,2,3),(2,3,1),(3,1,2)\}$,
	\begin{equation}\label{EQ:El}
		\dsum_{(i,j,k)\in \cE} u_iu_j(u_j-u_i)(\sigma+2\mu|u_k|)\delta u_k=0.
	\end{equation}
	Let $G$ be the Green function corresponding to the operator $-\nabla\cdot \gamma\nabla $ with the homogeneous Dirichlet boundary condition. We can then write~\eqref{EQ:El} as, using $\delta\gamma_{|\partial\Omega}=0$ as well as ${\delta u_j}_{|\partial\Omega}=0$,
	\[
		\dsum_{(i,j,k)\in \cE} u_iu_j(u_j-u_i)(\sigma+2\mu|u_k|)\int_\Omega\delta\gamma(\by) \nabla u_k(\by)\cdot\nabla G(\bx;\by)d\by=0.
	\]
	Take $u_k$ to be the complex geometric optics solution we constructed in Theorem~\ref{THM:CGO} with complex vector $\brho_k$, using the fact that $u_k \sim e^{\brho_k\cdot\bx}(1+\varphi_k)$ (and $\varphi_k$ decays as $|\brho_k|^{-1}$) and $\nabla u_k=u_k(\brho_k+O(1))$, we can rewrite the above equation as, for $|\brho_k|$ sufficiently large, 
	\[
		\dsum_{(i,j,k)\in\cE} u_iu_j(u_j-u_i)(\sigma+2\mu|u_k|)\int_\Omega\delta\gamma(\by) u_k(\by)\brho_k\cdot\nabla G(\bx;\by)d\by=0.
	\]
	Even though it is not necessary here, we can select $\brho_k$ such that $\Re\brho_k<0$ and $|\Re\brho_k|$ is sufficiently large to simplify the above equation further to
	\begin{equation}\label{EQ:LL}
		\int_\Omega \delta\gamma(\by) \bv(\bx;\by) \cdot\nabla G(\bx;\by)d\by=0.
	\end{equation}
	with $\bv$ the vector given by
	\begin{equation}\label{EQ:V Field}
		\bv= \dsum_{(i,j,k)\in\cE} \sigma(\bx)\Bigl(u_iu_j(u_j-u_i)\Bigr)(\bx) u_k(\by)\brho_k.
	\end{equation}
	We now need the following lemma.
	\begin{lemma}\label{LMMA:Inv}
		Let $\bv$ be such that: (i) there exists $\mathfrak c>0$ such that $|\bv|\ge \mathfrak c>0$ for a.e. $\bx\in\Omega$; and (ii) $\bv\in [W^{1,\infty}(\Omega)]^n$ at least. Then ~\eqref{EQ:LL} implies that $\delta\gamma\equiv 0$.
	\end{lemma}
	\begin{proof}
	Let $u$ be the solution to 
		\begin{equation}\label{EQ:Inv}
			-\nabla\cdot\gamma\nabla u - \nabla\cdot(\delta\gamma \bv) = 0,\quad 
			\mbox{in}\ \Omega, \qquad u=0,\quad\mbox{on}\ \partial\Omega
		\end{equation}
		Then 
		\[
			u(\bx)=\int_\Omega \delta\gamma(\by) \bv(\bx;\by) \cdot\nabla G(\bx;\by)d\by.
		\]
		Therefore ~\eqref{EQ:LL} implies that $u\equiv 0$. Therefore
		\begin{equation}
			-\nabla\cdot\delta\gamma\bv = 0,\quad \mbox{in}\ \Omega,\qquad \delta\gamma=0,\quad \mbox{on}\ \partial\Omega.
		\end{equation}
		This is a transport equation for $\delta\gamma$ that admits the unique solution $\delta\gamma\equiv 0$ with the vector field $\bv$ satisfying the assumed requirements; see for instance~\cite{BaRe-IP11,BoCr-SIAM06,CoCrRa-CPDE06,DiLi-AM89,Hauray-AIHP03} and references therein.
	\end{proof}
	It is straightforward to check that we can select $\{\brho_j\}_{j=1}^3$ such that the vector field $\bv$ defined in~\eqref{EQ:V Field} satisfies the requirement in Lemma~\ref{LMMA:Inv}. We then conclude that $\delta\gamma\equiv 0$.

	The conditions assumed on the background coefficients ensure the ellipticity of the system as proven in Theorem~\ref{THM:Ellipticity}. The stability result then follows from~\eqref{EQ:Stab D-N Sys}; see more discussions in~\cite{Bal-CM13} and references therein. Note that the simplification inthe last term in~\eqref{EQ:Stab Lin Sys gsm} is due to the fact that $\phi_k=\wt\phi_k=0$ when $4\le k\le J+3$. Note also that the following simplification can be made in~\eqref{EQ:Stab Lin Sys gsm}:
\[
	\sum_{i=1}^{2J}\Vert \cS_{i}-\wt\cS_{i}\Vert_{W^{\ell-s_i,2}(\Omega)}\le 2 \sum_{i=1}^{J}\Vert \delta H_{j}-\wt{\delta H}_j \Vert_{W^{\ell+2,2}(\Omega)}.
\]
The proof is complete.
\end{proof}

\section{Numerical simulations}
\label{SEC:Num}

We present in this section some preliminary numerical reconstruction results using synthetic internal data. We restrict ourselves to two-dimensional settings only to simplify the computation. The spatial domain of the reconstruction is the square $\Omega = (−1, 1)^2$. 
All the equations in $\Omega$ are discretized with a first-order finite element method on triangular meshes. In all the simulations in this section, reconstructions are performed on a finite element mesh with about $6000$ triangular elements. The nonlinear system resulted from the discretization of the diffusion equation~\eqref{EQ:Diff TP} is solved using a quasi-Newton method as implemented in~\cite{ReBaHi-SIAM06}.

To generate synthetic data for inversion, we solve~\eqref{EQ:Diff TP} using the true coefficients. We performed reconstructions using both noiseless and noisy synthetic data.
For the noisy data, we added random noise to the data by simply multiplying each datum by $(1 + \sqrt{3}\epsilon \times 10^{-2} \texttt{random})$ with \texttt{random} a uniformly distributed random variable taking values in $[−1, 1]$, $\epsilon$ being the noise level (i.e. the size of the variance in percentage).

We will focus on the reconstruction of the absorption coefficients $\sigma$ and $\mu$. We present reconstruction results from two different numerical methods. 

\paragraph{Direct Algorithm.} The first method we use is motivated from the method of proofs of Propositions~\ref{PROP:Stab Single Coeff} and ~\ref{PROP:Stab Two Coeff}. When we have $J\geq 2$ data sets $\{H_j\}_{j=1}^J$ from $J$ illuminations $\{g_j\}_{j=1}^J$, we first reconstruct, for each $j$, $u_j^*$ as the solutions to
\begin{equation*}
-\nabla\cdot(\gamma\nabla u_j^*) = -\frac{H_j^*}{\Gamma}\quad \text{in}\ \Omega,\qquad 
u_j^*= g_j\quad \text{on} \ \partial\Omega .
\end{equation*}
We then reconstruct $\sigma + \mu |u_j^*|=\frac{H_j}{\Gamma u_j^*}$. Collecting this quantity from all data, we have, for each point $\bx\in\Omega$,
\begin{equation*}
\begin{pmatrix}
1 & |u_1^*|\\
\vdots & \vdots \\
1 & |u_J^*|
\end{pmatrix}
\begin{pmatrix}
\sigma\\ \mu
\end{pmatrix}
=
\begin{pmatrix}
\frac{H_1^*}{\Gamma u_1^*}\\
\vdots \\
\frac{H_J^*}{\Gamma u_J^*}
\end{pmatrix}.
\end{equation*}
We then reconstruct $(\sigma, \mu)$ by solving this small linear system, in least square sense, at each point $\bx\in\Omega$. Therefore, the main computational cost of this algorithm lies in the numerical solution of the $J$ linear equations for $\{u_j^*\}$.

\paragraph{Least-Square Algorithm.} The second reconstruction method that we will use is based on numerical optimization. This method searches for the unknown coefficient by minimizing the objective functional
\begin{equation}
\Phi (\sigma,\mu) \equiv \frac{1}{2} \sum_{j=1}^{J}\int_{\Omega} (\Gamma\sigma u_j + \Gamma \mu |u_j|u_j - H_j^*)^2 d\bx + \kappa R(\sigma,\mu) ,\label{eq:data_misfit}
\end{equation}
where we use the functional $R(\sigma,\mu) = \frac{1}{2}\left(\int_{\Omega}|\nabla \sigma|^2 d\bx + \int_{\Omega}|\nabla \mu|^2 d\bx \right)$ together with the parameter $\kappa$ to add regularization mechanism in the reconstructions. We use the BFGS quasi-Newton method that we developed in~\cite{ReBaHi-SIAM06} to solve this minimization problem. It is straightforward to check, following Proposition~\ref{THM:Differentiability}, that the gradient of the functional $\Phi(\sigma,\mu)$ with respect to $\sigma$ and $\mu$ are given respectively by
\begin{align}
\Phi'_{\sigma} [\sigma,\mu](\delta \sigma) & =  \int_{\Omega}\left\{ \sum_{j=1}^{J} \Big[z_j \Gamma u_j+ v_j u_j\Big]\delta\sigma + \kappa \nabla \sigma \cdot \nabla \delta\sigma \right\}d\bx\\
\Phi'_{\mu} [\sigma,\mu](\delta\mu) & =  \int_{\Omega}\left\{ \sum_{j=1}^{J}\Big[z_j \Gamma |u_j|u_j + v_j |u_j|u_j\Big]\delta\mu + \kappa\nabla \mu \cdot \nabla\delta\mu\right\}d\bx
\label{eq:derivative}
\end{align}
where $z_j=\Gamma(\sigma u_j + \mu |u_j|u_j) - H_j^*$ and $v_j$ solves
\begin{equation}\label{EQ:Adj}
-\nabla \cdot \gamma \nabla v_j + (\sigma+ 2\mu |u_j|) v_j = - z_j \Gamma (\sigma + 2\mu |u_j|),\quad \text{in}\ \Omega,\qquad v_j = 0,\quad \text{on} \ \partial \Omega
\end{equation} 
Therefore, in each iteration of the optimization algorithm, we need to solve $J$ semilinear diffusion equations for $\{u_j\}_{j=1}^J$ and then $J$ adjoint linear elliptic equations for $\{v_j\}_{j=1}^J$ to evaluate the gradients of the objective function with respect to the unknowns.
\begin{figure}[htbp]
\centering
\includegraphics[width=0.3\textwidth]{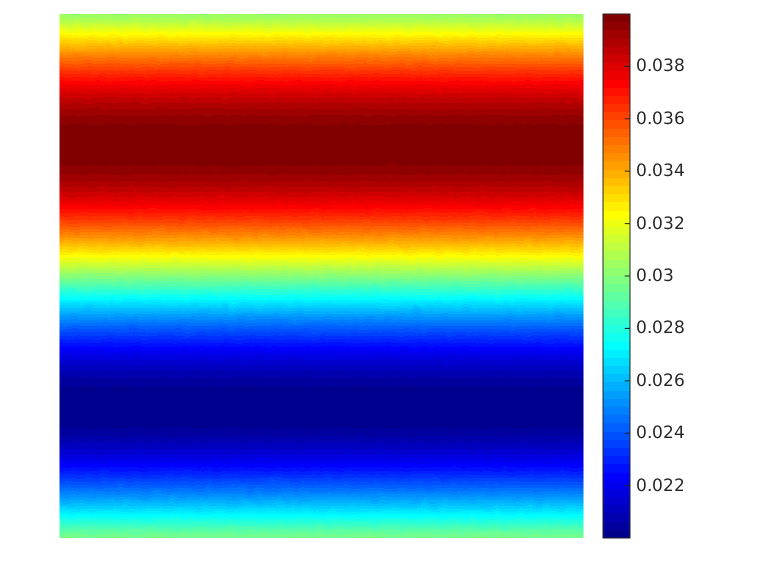}
\includegraphics[width=0.3\textwidth]{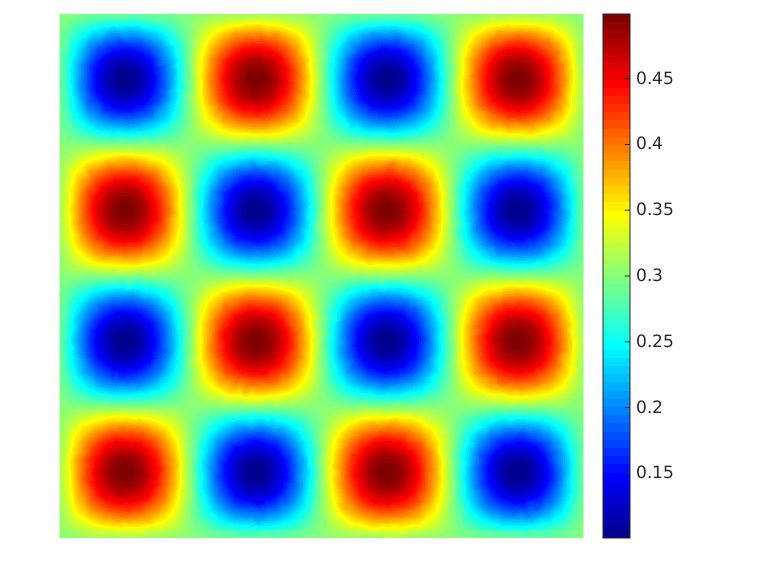}
\includegraphics[width=0.3\textwidth]{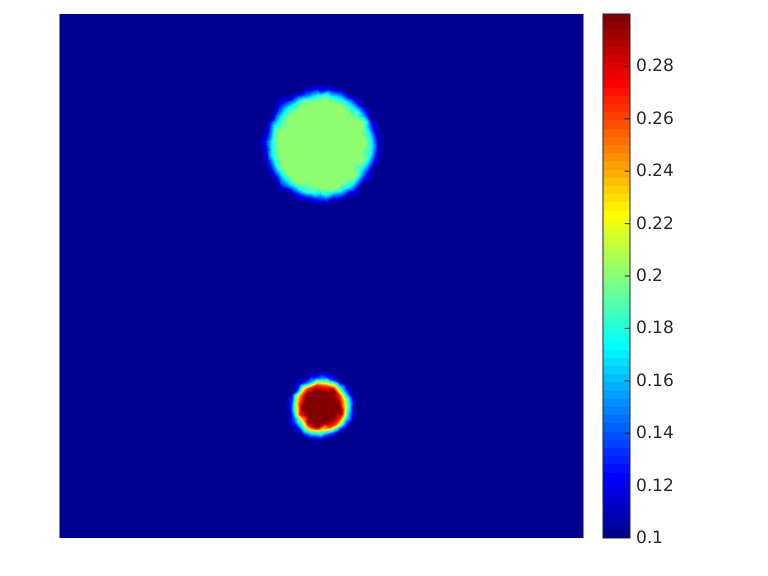}
\caption{The true coefficients, $\gamma$ (left), $\sigma$ (middle), $\mu$ (right), used to generate synthetic data for the reconstructions.}
\label{FIG:True Coeff}
\end{figure}

\begin{figure}[htbp]
\centering
\includegraphics[width=0.23\textwidth]{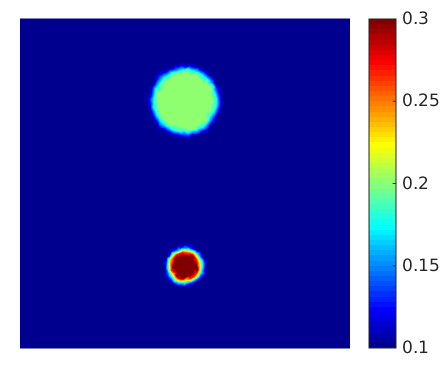}
\includegraphics[width=0.23\textwidth]{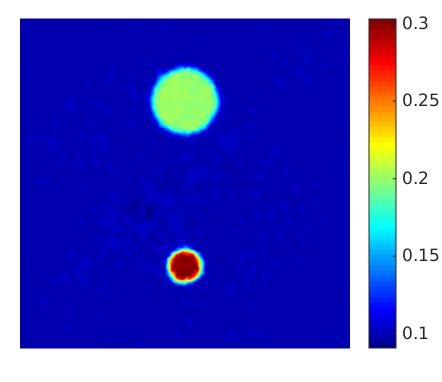}
\includegraphics[width=0.23\textwidth]{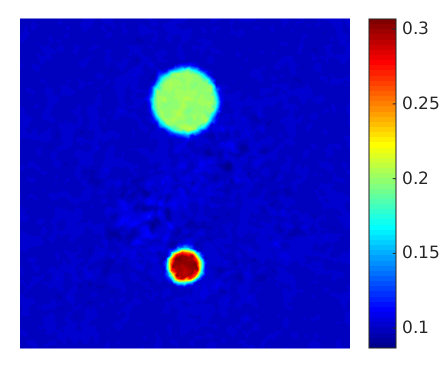}
\includegraphics[width=0.23\textwidth]{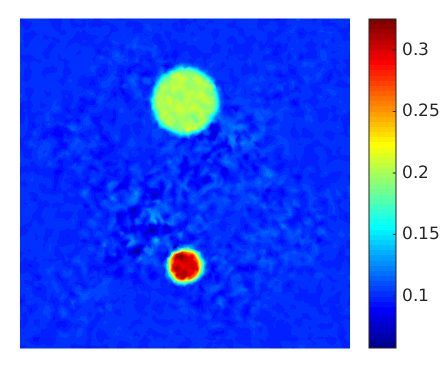}
\caption{The absorption coefficient $\mu$ reconstructed using synthetic data containing different levels ($\epsilon=0,1,2,5$ from left to right) of noises. The \emph{Direct Algorithm} is used in the reconstructions.}
\label{FIG:Mu Direct}
\end{figure}
\paragraph{Experiment I.} We start with a set of numerical experiments on the reconstruction of the two-photon absorption coefficient $\mu$ assuming that the single-photon absorption coefficient $\sigma$ is known. We use data collected from four different sources $\{g_j\}_{j=1}^4$, $\{H_j\}_{j=1}^4$. We perform reconstructions using the \emph{Direct Algorithm}. In Fig.~\ref{FIG:Mu Direct} we show the reconstruction results from noisy synthetic data with noise levels $\epsilon=0$, $\epsilon=1$, $\epsilon=2$, and $\epsilon=5$. The true coefficients used to generate the data are shown in Fig.~\ref{FIG:True Coeff}.

To measure the quality of the reconstruction, we use the relative $L^2$ error. This error is defined as the ratio between (i) the $L^2$ norm of the difference between the reconstructed coefficient and the true coefficient and (ii) the $L^2$ norm of the true coefficient, expressed in percentage. The relative $L^2$ errors in the reconstructions of $\mu$ in Fig~\ref{FIG:Mu Direct} are $0.00\%$, $2.45\%$,$4.98\%$, and $12.23\%$ for $\epsilon=0$, $\epsilon=1$, $\epsilon=2$ and $\epsilon=5$ respectively.

\begin{figure}[htbp]
\centering
\includegraphics[width=0.23\textwidth]{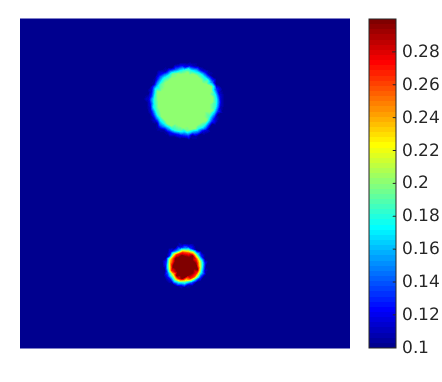}
\includegraphics[width=0.23\textwidth]{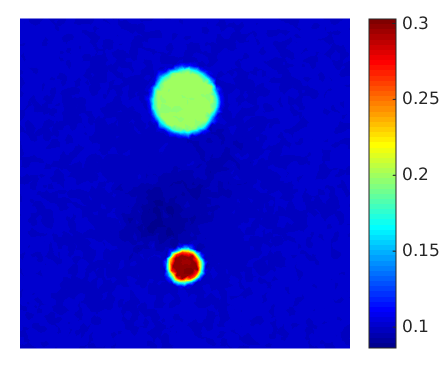}
\includegraphics[width=0.23\textwidth]{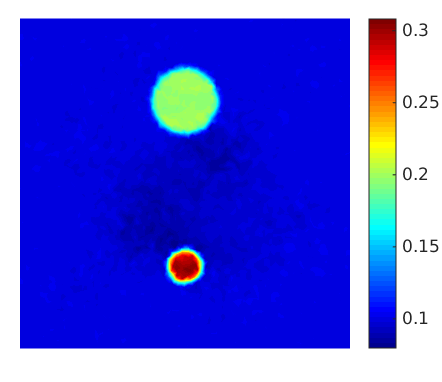}
\includegraphics[width=0.23\textwidth]{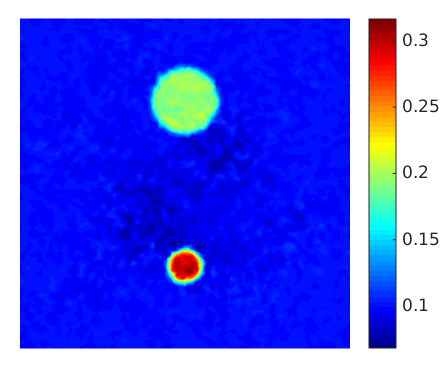}
\caption{Same as in Fig.~\ref{FIG:Mu Direct} except that the reconstructions are performed with the \emph{Least-Square Algorithm}.}
\label{FIG:Mu Minimiz}
\end{figure}
\paragraph{Experiment II.} One of the main limitations on the \emph{Direct Algorithm} is that it requires the use of illumination sources that are positive everywhere on the boundary. This is difficult to implement in practical applications. The \emph{Least-Square Algorithm}, however, does not have such requirement on the optical sources (but it is computationally more expensive). Here we repeat the simulations in Experiment I with the \emph{Least-Square Algorithm}. The reconstruction results are shown in Fig~\ref{FIG:Mu Minimiz}. We observe that, with the same (not exactly the same since the realizations of the noise are different) data sets, the reconstructions from the two different algorithms are of very similar quality. The relative $L^2$ errors for the reconstructions in Fig~\ref{FIG:Mu Minimiz} are $0.00\%$, $2.44\%$,$4.62\%$, and $9.36\%$ respectively for the four cases.

\begin{figure}[htbp]
\centering
\includegraphics[width=0.23\textwidth]{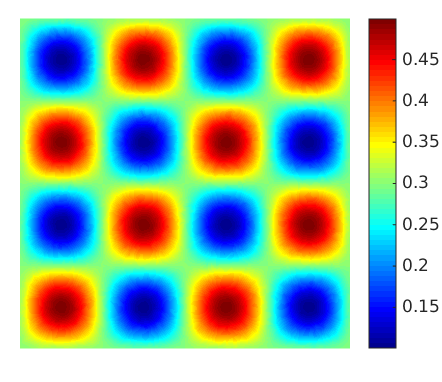}
\includegraphics[width=0.23\textwidth]{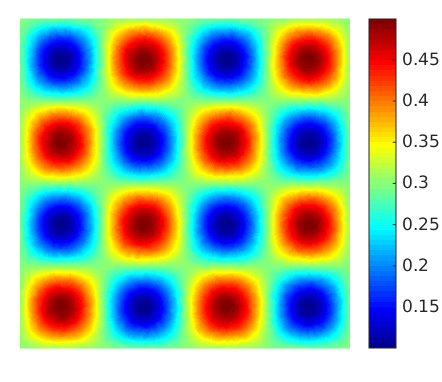}
\includegraphics[width=0.23\textwidth]{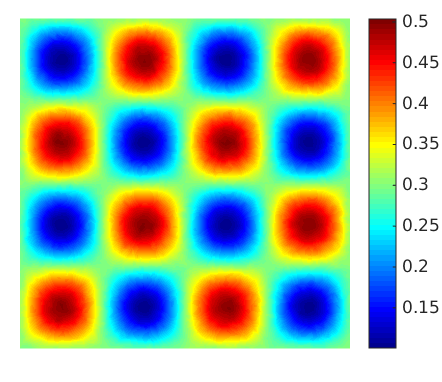}
\includegraphics[width=0.23\textwidth]{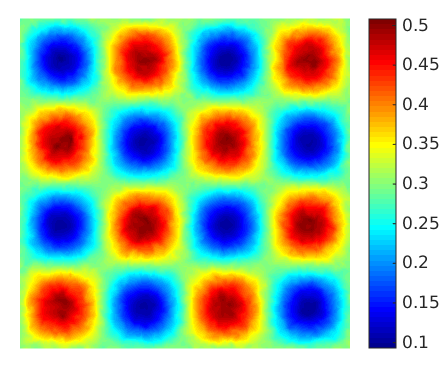}\\
\includegraphics[width=0.23\textwidth]{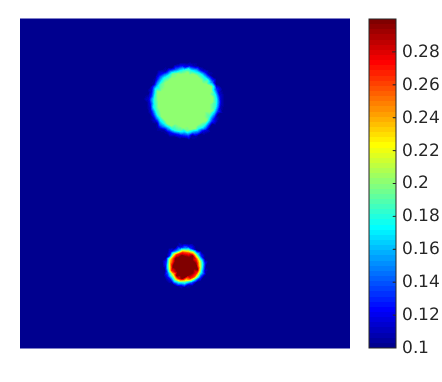}
\includegraphics[width=0.23\textwidth]{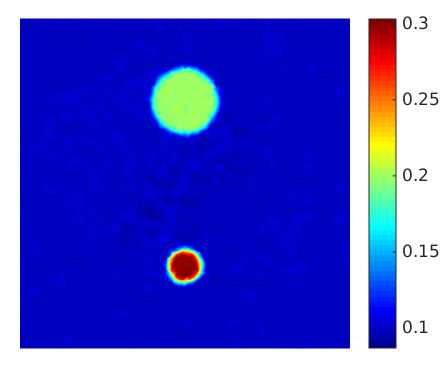}
\includegraphics[width=0.23\textwidth]{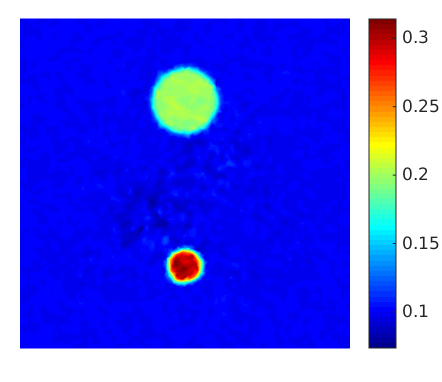}
\includegraphics[width=0.23\textwidth]{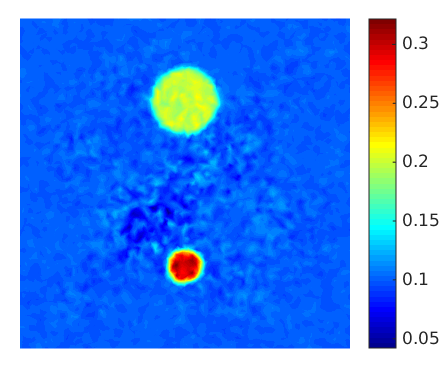}
\caption{The absorption coefficient pair $\sigma$ (top row) and $\mu$ (bottom row) reconstructed using the \emph{Direct Algorithm} with data at different noise levels ($\epsilon=0,1,2,5$ from left to right).}
\label{FIG:Sigma Mu Direct}
\end{figure}
\paragraph{Experiment III.} In the third set of numerical experiments, we study the simultaneous reconstructions of the single-photon and two-photon absorption coefficients, $\sigma$ and $\mu$. We again use data collected from four different sources. In Fig.~\ref{FIG:Sigma Mu Direct}, we show the reconstructions from data containing different noise levels using the \emph{Direct Algorithm}. The relative $L^2$ error in the reconstructions of $(\sigma, \mu)$ are $(0.00\%,0.00\%)$, $(0.79\%,2.76\%)$,$(1.56\%,5.55\%)$, and $(3.91\%,13.71\%)$ respectively for data with noise levels $\epsilon=0$, $\epsilon=1$, $\epsilon=2$ and $\epsilon=5$.

\begin{figure}[htbp]
\centering
\includegraphics[width=0.23\textwidth]{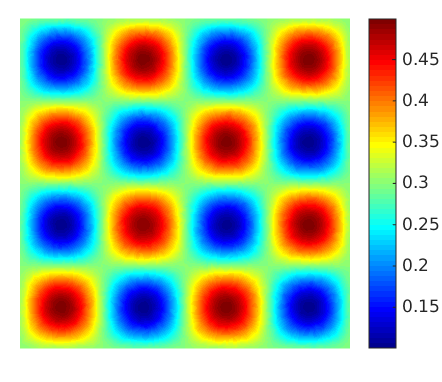}
\includegraphics[width=0.23\textwidth]{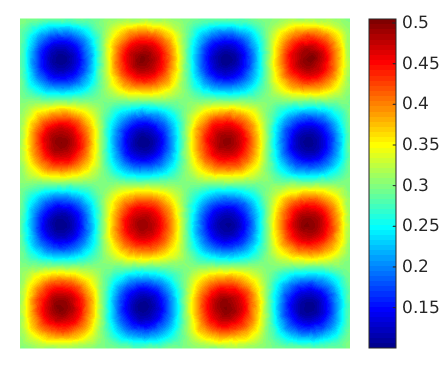}
\includegraphics[width=0.23\textwidth]{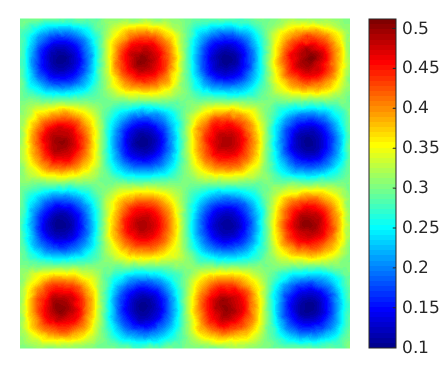}
\includegraphics[width=0.23\textwidth]{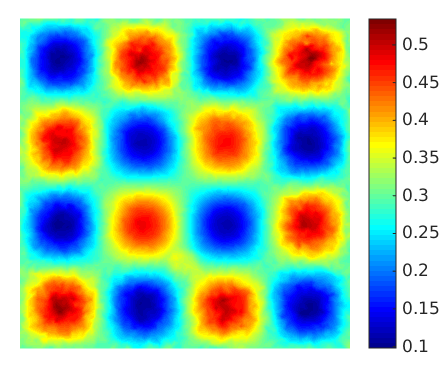}\\
\includegraphics[width=0.23\textwidth]{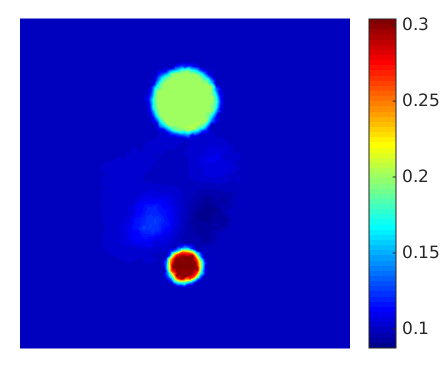}
\includegraphics[width=0.23\textwidth]{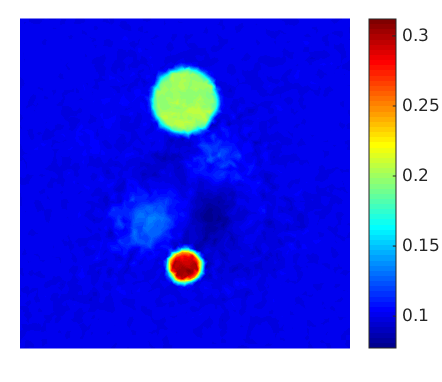}
\includegraphics[width=0.23\textwidth]{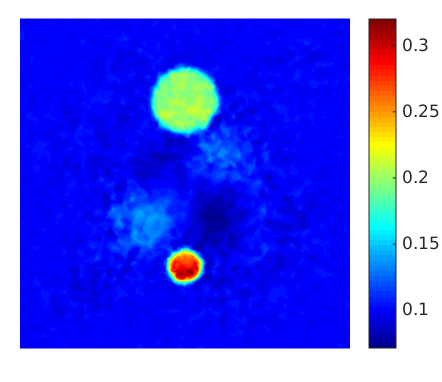}
\includegraphics[width=0.23\textwidth]{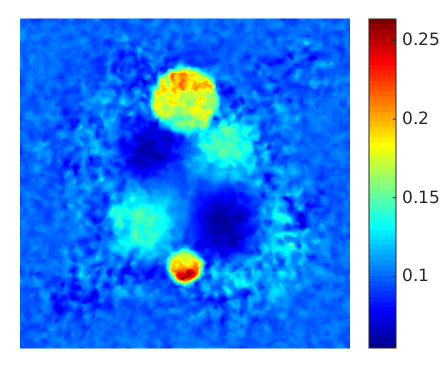}
\caption{The same as in Fig.~\ref{FIG:Sigma Mu Direct} except that the reconstructions are performed using the \emph{Least-Square Algorithm}.}
\label{FIG:Sigma Mu Minimiz}
\end{figure}
\paragraph{Experiment IV.} We now repeat the simulations in Experiment III with the \emph{Least-Square Algorithm}. The results are shown in Fig.~\ref{FIG:Sigma Mu Minimiz}. The relative $L^2$ errors in the reconstructions are now $(0.22\%,2.38\%)$, $(1.21\%,6.43\%)$,$(2.34\%,10.98\%)$, $(5.64\%,22.06\%)$ respectively for data with noise levels $\epsilon=0$, $\epsilon=1$, $\epsilon=2$, and $\epsilon=5$. The quality of the reconstructions is slightly lower than, but still comparable to, that in the reconstructions in Experiment III.

We observe from the above simulation results that, in general, the quality of the reconstructions is very high. When we have the illumination sources that satisfy the positivity requirement on the whole boundary of the domain, the \emph{Direct Algorithm} provides an efficient and robust reconstruction method. The \emph{Least-Square Algorithm} is less efficient but is as robust in terms of the quality of the reconstructions. The reconstructions with the \emph{Least-Square Algorithm} are done for a fixed regularization parameter that we selected by a couple of trial-error test. It is by no means the optimal regularization parameter that can be selected through more sophisticated algorithms~\cite{EnHaNe-Book96}. However, this is an issue that we think is not important at the current stage of this project. Therefore, we did not pursue further in this direction.

We also observe that the performances of the reconstructions of $\sigma$ and $\mu$ are different in both the direct and the least-square methods: the reconstructions of $\sigma$ is visually better than those of $\mu$ in general. This phenomenon is not manifested in our current stability results. Further analysis, for instance to refine the stability results in Section~\ref{SEC:Absorption}, are needed to fully understand this difference in numerical stability.

\section{Conclusion and further remarks}
\label{SEC:Concl}

We studied in this paper inverse problems in quantitative photoacoustic tomography with two-photon absorption. We derived uniqueness and stability results in the reconstruction of single-photon and two-photon absorption coefficients, and proposed explicit reconstruction methods in this case with well-selected illumination sources. We also studied the inverse problem of reconstructing the diffusion coefficient in addition to the absorption coefficients and obtained partial results on the uniqueness and stability of the reconstructions for the linearized problem. We presented some numerical studies based on the explicit reconstruction procedures as well as numerical optimization techniques to demonstrate the type of quality that can be achieved in reasonably controlled environments (where noise strength in the data is moderate).

Our focus in this paper is to study the mathematical properties of the inverse problems. There are many issues that have to be address in the future. Mathematically, it would be nice to generalize the uniqueness and stability results in Section~\ref{SEC:MultiPara}, on multiple coefficient reconstructions in linearized settings, to the fully nonlinear problem. Computationally, detailed numerical analysis, in three-dimensional setting, need to be performed to quantify the errors in the reconstructions in practically relevant scenarios. It is especially important to perform reconstructions starting from acoustic data directly, following for instance the one-step reconstruction strategy developed in~\cite{DiReVa-IP15}, to see how sensitive the reconstruction of the two-photon absorption coefficient is with respect to noise in the acoustic data. On the modeling side, it is very interesting to see if the current study can be generalized to radiative transport type models for photon propagation.

\section*{Acknowledgments}

We would like to the anonymous referees for their useful comments that help us improve the quality of the paper. This work is partially supported by the National Science Foundation through grants DMS-1321018 and DMS-1620473.

\section*{Appendix A: Terminologies in overdetermined elliptic systems}

We recall here, very briefly, some terminologies and notations related to overdetermined linear elliptic systems, following the presentation in~\cite{Bal-CM13,WiSc-IP15}. Let $M, \wt M, N$ be three positive integers such that $M>N$. we consider the following system of $M$ differential equations for $N$ variables $\{v_1,\cdots,v_N\}$ with $\wt M$ boundary conditions:
\begin{eqnarray}
\label{EQ:Elliptic Sys}
\cA(\bx,D) v &=& \cS, \quad \text{in}\ \Omega \\
\label{EQ:Elliptic Sys BC}
\cB(\bx,D)v &=& \phi, \quad \text{on}\ \partial \Omega
\end{eqnarray}
Here $\cA(\bx,D)$ is a matrix differential operator whose $(i, j)$ element, denoted by $\cA_{ij}(\bx, D)$ ($1\le i\le M$, $1\le j\le N$), is a polynomial in $D$ for any $\bx\in\Omega$. $\cB(\bx,D)$ is a matrix differential operator whose $(k, \ell)$ element, denoted by $\cB_{k\ell}(\bx, D)$ ($1\le k\le \wt M$, $1\le \ell \le N$), is a polynomial in $D$ for any $\bx\in\partial\Omega$.

We associate an integer $s_i$ ($1\le i\le M$) to row $i$ of $\cA$ and an integer $t_j$ to column $j$ ($1\le j\le N$) of $\cA$.
\begin{definition}
We call the integers $\{s_i\}_{i=1}^M$ and $\{t_j\}_{j=1}^N$ the Douglis-Nirenberg numbers associated to $\cA$ if: (a) $s_i\le 0$, $1\le i\le M$; (b) when $s_i+t_j\ge 0$, the order of $\cA_{ij}(\bx,D)$ is not greater than $s_i+t_j$; and (c) when $s_i+t_j<0$, $\cA_{ij}(\bx,D)=0$.
\end{definition}

\begin{definition}
The principal part of $\cA$, denoted by $\cA_0$, is defined as the part of $\cA$ such that the degree of $\cA_{0,ij}(\bx,D)$ is exactly $s_i+t_j$.
\end{definition}

We say that $\cA$ is elliptic, in the sense of Douglis-Nirenberg, if the matrix $\cA_0(\bx,\bxi)$ is of rank $N$ for all $\bxi \in\bbS^{n-1}$ ($\bbS^{n-1}$ being the unit sphere in $\bbR^n$) and all $\bx\in \Omega$.

Let $b_{k\ell}$ be the order of $\cB_{k\ell}$ and define
\begin{equation}\label{EQ:Eta}
	\eta_k=\max_{1\le \ell\le N}(b_{k\ell}-t_\ell),\ \ 1\le k\le \wt M.
\end{equation}
\begin{definition}
The principal part of $\cB$, denoted by $\cB_0$, is defined as the part of $\cB$ such that the order of $\cB_{0,k\ell}$ is exactly $\eta_k+t_\ell$.
\end{definition}

Let $\cB_{0}(\bx,D)$ be the principal part of $\cB$. Fix $\by\in \partial \Omega$, and let $\bnu$ be the \emph{inward} unit normal vector at $\by$. Let $\boldsymbol \zeta \in \bbS^{n-1}$ be a vector such that $\boldsymbol\zeta\cdot\bnu=0$ and $|\boldsymbol\zeta|\neq 0$. We consider on the half line $\by+z\bnu, z>0$ the system of ordinary equations
\begin{eqnarray}
\label{EQ:Lop A} \mathcal{A}_{0}(\by, \mathfrak i\boldsymbol\zeta + \bnu \frac{d}{dz}) \wt{u}(z) &= 0,\quad z>0,\\
\label{EQ:Lop B} \mathcal{B}_{0}(\by, \mathfrak i\boldsymbol\zeta + \bnu \frac{d}{dz}) \wt{u}(z) &= 0,\quad z=0.
\end{eqnarray}
\begin{definition}\label{DEF:Lop BC}
	If for any $\by \in \partial \Omega$, the only solution to the system~\eqref{EQ:Lop A}-\eqref{EQ:Lop B} such that $\wt{u}(z)\to 0$ as $z\to \infty$ is $\wt u\equiv 0$, then we say that $(\mathcal{A},\mathcal{B})$ satisfies the Lopatinskii criterion.
\end{definition}

It is well-established that~\cite{Bal-CM13,Solonnikov-JSM73,WiSc-IP15} when $(\cA, \cB)$ satisfies the Lopatinskii criterion, the system~\eqref{EQ:Elliptic Sys}-\eqref{EQ:Elliptic Sys BC} can be solved up to possibly a finite dimensional subspace. Moreover, a general \emph{a priori} stability estimate can be established for the system. Define the function space
\begin{equation*}
	\cW_\ell = W^{\ell-s_1,2}(\Omega) \times \ldots \times W^{\ell-s_{M},2}(\Omega) \times W^{\ell-\eta_1-\frac{1}{2},2}(\partial \Omega)\times \ldots W^{\ell-\sigma_{\wt M}-\frac{1}{2},2}(\partial \Omega),
\end{equation*}
for some $\ell>n+\frac{1}{2}$. Then it can be shown that, if $(\cS,\phi)\in\cW_\ell$,
\begin{equation}\label{EQ:Stab D-N Sys}
	\sum_{j=1}^{N}\Vert v_j \Vert_{W^{\ell+t_j,2}(\Omega)} \leq C \Bigl(\sum_{i=1}^{M}\Vert \cS_{i}\Vert_{W^{\ell-s_i,2}(\Omega)} + \sum_{i=1}^{\wt M} \Vert \phi_i \Vert_{W^{\ell-\eta_i-\frac{1}{2},2}(\partial \Omega)}\Bigr) + \wt C \sum_{t_j>0} \Vert v_j \Vert_{L^2(\Omega)},
\end{equation}
provided that all the quantities involved are regular enough. The last term in the estimate can be dropped when uniqueness of the solution can be proven. More details of this theory can be found in~\cite{Bal-CM13} and references therein.

\section*{Appendix B: CGO solutions to equation~\eqref{EQ:Diff TP}}

This appendix is devoted to the construction of complex geometric optics (CGO) solutions~\cite{SyUh-AM87,Uhlmann-IP09} to our model equation~\eqref{EQ:Diff TP}. We restrict the construction to the three-dimensional setting ($n=3$). We start by revisiting CGO solutions to the classical diffusion problem that was first developed in~\cite{SyUh-AM87}:
\begin{equation}\label{EQ:Diff SP}
- \nabla \cdot (\gamma \nabla u) + \sigma u = 0, \quad \text{in}\ \Omega
\end{equation}
with the assumption that $\gamma \in \cC^2(\Omega)$ and $\sigma \in \cC^1(\Omega)$. Let $u_{*}$ be a solution to this equation, then the Liouville transform defined in~\eqref{EQ:Liouville} shows that $\tilde{u}_{*}=\sqrt{\gamma}u_{*}$ solves
\begin{equation}\label{EQ:Diff SP Liouville}
	\Delta \tilde{u}_{*} - \Bigl(\frac{\Delta \sqrt{\gamma}}{\sqrt{\gamma}} + \frac{\sigma}{\gamma}\Bigr)\tilde{u}_{*} = 0, \quad \mbox{in}\ \Omega.
\end{equation}
The following result is well-known.
\begin{theorem}[\cite{Bal-Notes12,SyUh-AM87,Uhlmann-IP09}]\label{THM:CGO SP}
Let $\gamma \in \cC^4(\Omega)$ and $\sigma \in \cC^2(\Omega)$. For any $\brho\in \bbC^{n}$ such that $\brho\cdot\brho = 0$ and $|\brho|$ sufficiently large, there is a function $g$ such that the solution to~\eqref{EQ:Diff SP Liouville}, with the boundary condition ${\tilde u_{*}}|_{\partial\Omega}= g$, takes the form
\begin{equation}\label{EQ:CGO}
	\tilde{u}_{*} = e^{\brho\cdot \bx}(1+\varphi(\bx)),
\end{equation}
with $\varphi(\bx)$ satisfying the estimate
\begin{equation}\label{EQ:CGO phi}
	|\brho| \Vert\varphi\Vert_{W^{2,2}(\Omega)} + \Vert\varphi\Vert_{W^{3,2}(\Omega)} \leq C \biggl\Vert \frac{\Delta \sqrt{\gamma}}{\sqrt{\gamma}} +\frac{\sigma}{\gamma} \biggr\Vert_{W^{2,2}(\Omega)}.
\end{equation}
\end{theorem}

The function $\tilde u_*$ is called a complex geometric optics solution to~\eqref{EQ:Diff SP Liouville} and 
\begin{equation}\label{EQ:CGO Diff SP}
	u_{*}=\gamma^{-1/2}e^{\brho\cdot \bx}(1+\varphi(\bx))
\end{equation}
is a complex geometric optics solution to ~\eqref{EQ:Diff SP}. With the regularity assumption on $\gamma$ and~\eqref{EQ:CGO phi}, it is easy to verify that
\begin{equation}\label{EQ:CGO Diff SP Grad}
	\nabla u_{*} \sim u_*(\brho+O(1)).
\end{equation}

We now show, using the Newton-Kantorovich method~\cite{Ortega-AMM68}, that we can construct a CGO solution to our semilinear diffusion model that is very close, in $W^{3,2}(\Omega)$, to $u_*$ for some $\brho$.
\begin{theorem}\label{THM:CGO}
Let $\gamma \in \cC^4(\Omega)$, $\sigma \in \cC^2(\Omega)$ and $\mu \in \cC^1(\Omega)$. Let $\brho\in \bbC^{n}$ be such that $\brho\cdot\brho = 0$ and $|\brho|$ sufficiently large. Assume further that $\brho$ and $\Omega$ satisfy
\begin{equation}\label{EQ:Assump rho}
-\wt\kappa|\brho| \leq \Re(\brho\cdot \bx) \leq -\kappa|\brho|,\ \ \forall \bx\in\bar\Omega,
\end{equation}
for some $0<\kappa < \wt \kappa<\infty$. Then, there exists a function $g$ such that the solution to~\eqref{EQ:Diff TP} takes the form 
\begin{equation}\label{EQ:CGO TP}
	u(\bx)=u_*(\bx)+v(\bx),
\end{equation}
with $v$ such that
\begin{equation}
	\|v\|_{W^{3,2}(\Omega)}\le c e^{-\kappa' \kappa|\brho|},
\end{equation}
for some constant $c$ and some $\kappa'\in(1,2)$.
\end{theorem}
\begin{proof}
We first observe that the assumption in~\eqref{EQ:Assump rho} allows us to bound the CGO solution $u_*$ and its gradient as
\begin{equation}\label{EQ:CGO Bound} 
	\|u_{*}\|_{L^\infty(\Omega)} \leq c e^{-\kappa|\brho|},
\end{equation}
\begin{equation}\label{EQ:CGO Bound Grad} 
	\|\nabla u_{*}\|_{L^\infty(\Omega)} \leq c e^{-\kappa|\brho|}(|\brho|+1).
\end{equation}

Let $\cP(\bx,D)$ be the differential operator defined in ~\eqref{EQ:Diff TP}. We verify that, with the assumptions on the coefficients involved, $\cP_{u_*}'$, the linearization of $\cP$ at $u_*$, $\cP_{u_*}': =-\nabla\cdot\gamma\nabla + \sigma + 2\mu|u_*|$, admits a bounded inverse as a linear map from $W^{3,2}(\Omega)$ to $W^{1,2}(\Omega)$.

We observe from the construction that $u_*$ is away from $0$. Therefore, there exists a constant $r>0$ such that the ball $B_r(u_*)$ (in the $W^{3,2}(\Omega)$ metric) contains functions that are away from $0$. Let $u_1\in B_r(u_*)$, $u_2\in B_r(u_*)$ and  $v\in W^{3,2}(\Omega)$ be given, we check that
\begin{equation}
	\cP_{u_1}' v-\cP_{u_2}' v = 2\mu(|u_1|-|u_2|)v.
\end{equation}
This leads to
\begin{equation}\label{EQ:Operator Bound A1}
	\|\cP_{u_1}' v-\cP_{u_2}' v\|_{L^2(\Omega)} = \|2\mu(|u_1|-|u_2|)v\|_{L^2(\Omega)}\le c\|u_1-u_2\|_{W^{3,2}(\Omega)} \|v\|_{W^{3,2}(\Omega)}.
\end{equation}
We can also bound $\|\nabla(\cP_{u_1}' v-\cP_{u_2}'v)\|_{L^2(\Omega)}$ as follows. We first verify that
\begin{multline}\label{EQ:Operator Bound B1}
\|\nabla(\cP_{u_1}' v-\cP_{u_2}'v)\|_{L^2(\Omega)}=2\|\nabla \big(\mu(|u_1|-|u_2|)v\big)\|_{L^2(\Omega)}\\
\leq 2\| \mu(|u_1|-|u_2|) \nabla v \|_{L^2(\Omega)} + 2\|\nabla\bigl(\mu(|u_1|-|u_2|) \bigr) v \|_{L^2(\Omega)}\\
\leq 2\|\mu(|u_1|-|u_2|)\|_{L^{\infty}(\Omega)} \|\nabla v\|_{L^{2}(\Omega)} + 2\|\nabla\bigl(\mu(|u_1|-|u_2|)\bigr) \|_{L^{\infty}(\Omega)} \|v \|_{L^{2}(\Omega)}\\
\le C \|v\|_{W^{3,2}(\Omega)} \Bigl(\Vert u_1-u_2 \Vert_{L^{\infty}(\Omega)} + \Vert \nabla\bigl(\mu(|u_1|-|u_2|\bigr)\Vert_{L^{\infty}(\Omega)}\Bigr).
\end{multline}
We then perform the expansion, using the fact that $|u_j|>0$ ($j=1,2$),
\begin{equation*}
\nabla\big(\mu(|u_1|-|u_2|)\bigr)=(|u_1|-|u_2|)\nabla\mu+\mu \Re\big(\dfrac{u_1}{|u_1|}\nabla(\bar u_1-\bar u_2)+(\dfrac{u_1}{|u_1|}-\dfrac{u_2}{|u_2|})\nabla\bar u_2 \big).
\end{equation*}
This gives us the bound
\begin{equation}\label{EQ:Operator Bound B2}
\|\nabla\big(\mu(|u_1|-|u_2|)\bigr)\|_{L^\infty(\Omega)}\le c_1 \|u_1-u_2\|_{L^\infty(\Omega)} + c_2\|\nabla(u_1-u_2)\|_{L^\infty(\Omega)}.
\end{equation}
We can then combine~\eqref{EQ:Operator Bound B1} with ~\eqref{EQ:Operator Bound B2} and use Sobolev embedding, for instance~\cite[Corollary 7.11]{GiTr-Book00}, to conclude that
\begin{equation}\label{EQ:Operator Bound A2}
\|\nabla(\cP_{u_1}' v-\cP_{u_2}'v)\|_{L^2(\Omega)}\le C \|u_1-u_2\|_{W^{3,2}(\Omega)} \|v\|_{W^{3,2}(\Omega)}.
\end{equation}

We then have, from the bounds in~\eqref{EQ:Operator Bound A1} and~\eqref{EQ:Operator Bound A2}, the following bound on the operator norm of $\cP_{u_1}'-\cP_{u_2}'$ by
\begin{equation}
	\|\cP_{u_1}'-\cP_{u_2}'\|_{\cL\big(W^{3,2}(\Omega), W^{1,2}(\Omega)\big)} \le c\|u_1-u_2\|_{W^{3,2}(\Omega)}.
\end{equation}
Let $w$ be the solution to $\cP_{u_*}'(\bx, D) w =\cP(\bx,D) u_*$ (such that $w =(\cP_{u_*}')^{-1}\cP(\bx,D) u_*$), that is,
\begin{equation}
	-\nabla\cdot\gamma\nabla w +(\sigma + 2\mu|u_*|)w = \mu|u_*|u_*,\ \ \mbox{in}\ \Omega, \qquad w=0,\ \ \mbox{on}\ \partial\Omega.
\end{equation}
It then follows from classical elliptic theory that
\begin{equation}
	\Vert (\cP_{u_*}')^{-1}\cP(\bx,D)u_{*} \Vert_{W^{3,2}(\Omega)} \le c\||u_*|u_*\|_{W^{1,2}(\Omega)}\le \tilde c e^{-2\kappa |\brho|}(|\brho|+1)\le \tilde{\tilde c} e^{-\kappa'\kappa|\brho|},
\end{equation}
for some $\kappa'\in(1,2)$, where the last step comes from the bounds in~\eqref{EQ:CGO Bound} and~\eqref{EQ:CGO Bound Grad}.

It then follows from the Newton-Kantorovich theorem~\cite{Ortega-AMM68} that, when $|\brho|$ is sufficiently large, there exists a solution to~\eqref{EQ:Diff TP} in the ball of radius $r'=\tilde {\tilde c}e^{-\kappa' \kappa |\brho|}$ centered at $u_{*}$, in $W^{3,2}(\Omega)$. The solution is of the form~\eqref{EQ:CGO TP}.
\end{proof}

Let us remark that the assumption~\eqref{EQ:Assump rho} on $\brho$ and $\Omega$ is feasible. Let $(\be_1, \be_2, \be_3)$ be a basis of $\bbR^3$. Since $\Omega$ is bounded, we can place $\Omega$ in a box in the first octant of the $(\be_1, \be_2, \be_3)$ system. Let us take $\Re(\brho)=a_1 \be_1+a_2 \be_2+a_3 \be_3$ for some $\{a_k\}_{k=1}^3$. We check that we can have $\tilde\kappa|\Re(\brho)|\le \Re(\brho) \le \kappa |\Re(\brho)|$ for some constant $\tilde\kappa$ and $\kappa$. If we further take $|\Im(\brho)|$ to be proportional to $|\Re(\brho)|$, we can have~\eqref{EQ:Assump rho}; see Section~\ref{SEC:MultiPara} for the choices of different $\brho$ in our analysis.


\end{document}